\documentclass{amsart}
\usepackage{mathpreamble}

\begin{document}

\title{K\"ahler geometry on total spaces of vector bundles over elliptic curves}

%\begin{comment}
\author{Hanyu Wu}
\address{School of Mathematical Sciences, Xiamen University, Xiamen, Fujian, 361005, China.}
\email{{19020230157192@stu.xmu.edu.cn}}
\author{Bo Yang}
%\begin{comment}
\thanks{The second named author is partially supported by the National Natural Science Foundation of China
under grant numbers: 11801475, 12141101, and 12271451, and the Natural Science Foundation of Fujian Province of China under the grant No.2019J05012.}    
%\end{comment}
\address{School of Mathematical Sciences, Xiamen University, Xiamen, Fujian, 361005, China.}
\email{{boyang@xmu.edu.cn}}    
%\end{comment}

\date{10/07/2026}

\begin{abstract}
We study function theory and K\"ahler geometry on total spaces of vector bundles over an elliptic curve. Our main result characterizes the vector bundles whose total spaces admit complete K\"ahler metrics with nonnegative bisectional curvature ($BI \geq 0$ for short). We also obtain explicit expressions for all complete K\"ahler metrics with $BI \geq 0$ on the total space of a line bundle over a complex torus. These results are related to a general conjecture concerning the structure of complete K\"ahler manifolds with $BI \geq 0$.

For rank-two vector bundles of degree zero over elliptic curves, we study holomorphic functions and automorphisms on the corresponding total spaces, and construct complete Gauduchon Hermitian metrics with vanishing Chern-Ricci curvature. These metrics are natural in the sense that the associated spaces of holomorphic functions of polynomial growth coincide with the ``polynomials'' on these total spaces.
\end{abstract}


\maketitle

\markleft{On total spaces of vector bundles over elliptic curves}
\markright{On total spaces of vector bundles over elliptic curves}

\setcounter{tocdepth}{1}
\tableofcontents

\begin{comment}
{\color{blue}

(1)  still working

}    
\end{comment}

\section{Introduction}

\subsection{Vector bundles over elliptic curves: two examples}
In this work, we are interested in the K\"ahler geometry of a special class of noncompact complex manifolds, namely those which are total spaces of vector bundles over an elliptic curve. We begin with two illustrative examples.

\subsubsection{Line bundles with degree zero}
Holomorphic line bundles over complex tori are well studied in the literature. For example, we consider an elliptic curve 
\begin{align}
\Sigma=\mathbb{C}/\Gamma, \ \ \ \   \Gamma=\mathbb{Z}\text{-span}\{1, \tau\} \ \ \text{for}\ \tau \in \mathbb{C}\ \text{with}\ \operatorname{Im}\tau>0.     \label{ecurve}
\end{align}
Therefore, $\pi_1(\Sigma)$ is a free abelian group generated by 
\begin{equation}
\gamma_1(t)=[z+t], \ \ \gamma_2(t)=[z+\tau t]\ \ \text{where}\  t \in [0, 1].   \label{funda_group}     
\end{equation}
Then the total space of any $L \in \operatorname{Pic}^{0} (\Sigma)$ can be written as
\begin{align}
X_{\theta}=\mathbb{C}^2/\pi_1(\Sigma), \ \ \text{with}\ \gamma_k(z, \xi)=(\gamma_k(z), (\alpha(\gamma_k))\xi),\ \text{for}\ k=1, 2.  \label{Xthetadef}
\end{align}
Here, we pick $\alpha \in \operatorname{Hom}(\pi_1(\Sigma), S^1)$ as
$\alpha(\gamma_k)=e^{2\pi i \theta_k}$ for some $\theta_k \in [0, 1)$ for $k \in \{1, 2\}$. It follows that 
\(
\sqrt{-1}(dz \wedge d\overline{z}+d\xi \wedge d\overline{\xi})
\), the standard Euclidean metric on $\mathbb{C}^2$, descends to $X_{\theta}$. Let $\omega$ be the corresponding K\"ahler form on $X_{\theta}$ for $\theta=(\theta_1, \theta_2)$. Note that $(X_{\theta}, \omega)$ is locally a metric product, and it splits globally only if $\theta=(0, 0)$. The corresponding function theory on $X_{\theta}$ depends on the arithmetic properties of $\theta$. For example, there are no nonconstant holomorphic functions on $X_{\theta}$ if either $\theta_1$ or $\theta_2$  is irrational. 

\begin{comment}
Indeed, any holomorphic function on $X_{\theta}$ lifts to a $f \in \mathcal{O}(\mathbb{C}^2)$ which satisfies $f(z, \xi)=f(z+1, e^{2\pi i\theta_1 } \xi)=f(z+\tau, e^{2\pi i \theta_2} \xi)$. Consider the power series expansion of $f$
\[
f(z, \xi)=f_0(z)+\sum_{k=1}^{\infty} f_{k}(z){\xi}^{k}.
\] Then we have $f_{k}(z+1)=e^{-2\pi i\theta_1 k}f_{k}(z)$ and $f_{k}(z+\tau)=e^{-2\pi i\theta_2 k}f_{k}(z)$ for any integer $k \geq 0$. It follows from Liouville's theorem that $f_0(z)$ is constant, and $f_{k}=0$ for any $k \geq 1$ once any of $\theta_i$ is irrational.    
\end{comment}

\subsubsection{An indecomposable rank $2$ bundle of degree zero} \label{discuss_E(2, 0)}
We review a rank $2$ bundle which appeared in Atiyah \cite{Atiyah} and Demailly-Peternell-Schneider \cite[Example 1.17]{DPS94}. In the latter, an example of a compact K\"ahler manifold with nef anticanonical line bundle is constructed. Moreover, it cannot admit any K\"ahler metric with nonnegative Ricci curvature. The example is the projective bundle $\mathbb{P}(E)$ for some rank $2$ vector bundle $E$ over an elliptic curve. It is the rank $2$ vector bundle $E$ that we would like to discuss.

Let $\Sigma$ be the elliptic curve defined in (\ref{ecurve}). We consider a representation $\rho: \pi_1(\Sigma) \rightarrow \operatorname{GL}(2, \mathbb{C})$ by
\[
\rho(\gamma_1)=\begin{pmatrix}  1 & 0 \\ 0 & 1 \end{pmatrix},\ \ \  \rho(\gamma_2)=\begin{pmatrix}  1 & 1 \\ 0 & 1 \end{pmatrix}. 
\]
The vector bundle $E$ which is associated to $\rho$ is defined as
\begin{align}
E=(\mathbb{C} \times \mathbb{C}^2) /\pi_1(\Sigma),\ \ \ \text{with}\ \gamma_i(z, z_1, z_2)=(\gamma_i(z), \rho(\gamma_i)(z_1, z_2))\ \text{for}\ 1\leq i\leq 2.  \label{Edef}      
\end{align}
In fact, $E$ is exactly the canonical rank $2$ bundle $F_2$ of degree $0$ in Atiyah \cite[Theorem 5]{Atiyah}. It is proved in \cite{Atiyah} and \cite{DPS94} that $E$ is indecomposable with $h^0(\Sigma, E)=1$, and it is an extension of two trivial line bundles. It is further shown in {\cite{DPS94}} that $E$ defined in (\ref{Edef}) is numerically flat (i.e. both $E$ and $\operatorname{det}(E^{\ast})$ are nef), and cannot be Griffiths semipositive.

\begin{comment}
 \begin{proof}[Proof of Lemma \ref{fact1}]
The transition function of the vector bundle $E$ is locally constant and determined by $\rho(\gamma_1)$ and $\rho(\gamma_2)$. The $(1, 1)$ entry of such matrices is always $1$ which corresponds to a trivial line bundle $F$. Indeed, $F$ can be expressed as
$(\mathbb{C} \times \mathbb{C} \times \{0\}) / (\Gamma, \rho)$.
The corresponding $(2, 2)$ entry of any transition function is again $1$, which corresponds to a trivial quotient bundle $Q=E/F$. If (\ref{exact}) splits then we get $E=\mathcal{O} \oplus \mathcal{O}$ and $h^0(\Sigma, E) \geq 2$. 

In contrast, we will show $h^0(\Sigma, E)=1$. Indeed, any global section of $E$ lifts a holomorphic map $\eta: z \in \mathbb{C} \rightarrow (\eta_0, \eta_1, \eta_2) \in \mathbb{C} \times \mathbb{C}^2$. so that there exists integers $p, q, r, s$ so that
\begin{align*}
&\eta(z+1)=(\eta_0(z)+p+q\tau, \eta_1(z)+q\eta_2(z), \eta_2(z)),\\ 
&\eta(z+\tau)=(\eta_0(z)+r+s\tau, \eta_1(z)+s\eta_2(z), \eta_2(z)).    
\end{align*}
Note that the first component $\eta_0$ descends to an identity map $\Sigma \rightarrow X \rightarrow \Sigma$.
It follows that $\eta_0(z)=z \,\text{mod}\, \Gamma$. Combined with the above equations, we get $p=s=1$ and $r=q=0$.

Next we note $\eta_2(z)$ is constant as it is doubly periodic, Assume that $\eta_2(z)=c$, we get
\[
\eta_1(z+1)=\eta_1(z)+q\,c=\eta_1(z),\ \ \eta_1(z+\tau)=\eta_1(z)+s\,c.
\] By Liouville's theorem, we have $\eta_1(z)$ is constant and $\eta_2(z)=c=0$. Hence $h^0(\Sigma, E)=1$.
\end{proof}   
\end{comment}

We observe the following characterization of holomorphic functions on its total space $X$.

\begin{proposition}[a special case of Proposition \ref{Xfunction}]\label{holoDPS}
    Let $X$ be the total space of $E$ defined in (\ref{Edef}). Then any $f \in \mathcal{O}(X)$ descends from a holomorphic function $\widetilde{f}(z_2) \in \mathcal{O}(\mathbb{C})$ in the sense of \eqref{Edef}. Here $(z, z_1, z_2)$ denotes the global coordinates on $\mathbb{C}^3$ as in Proposition \ref{Xfunction}.
\end{proposition}

In Proposition \ref{holoCala}, we apply the above result to investigate the growth estimates of any $f \in \mathcal{O}(X)$ with respect to a natural class of complete K\"ahler metrics on $X$.

\subsubsection{Guiding questions}
Vector bundles over compact Riemann surfaces have been extensively studied in the context of algebraic geometry. Given any vector bundle $E$ of rank $r$ over a compact Riemann surface $M$, we may write $E=\oplus_{i=1}^p E_i$ as a direct sum of indecomposable vector bundles. Then a classical result of Weil \cite{Weil} states that $E$ is associated to a representation $\rho: \pi_1(M) \rightarrow \operatorname{GL}(r, \mathbb{C})$ if and only if $\operatorname{deg}(E_i)=0$ for any $1 \leq i \leq p$. Moreover, indecomposable vector bundles over elliptic curves were classified by Atiyah \cite{Atiyah}; see also Oda \cite{Oda} and Iena \cite{Iena}. Moduli spaces of stable (semistable) vector bundles on compact Riemann surfaces of genus $g \geq 1$ have been extensively studied since Narasimhan-Seshadri \cite{NS1965} and Tu \cite{Tu}. In view of these algebraic results and two examples discussed above, we would like to study complex differential geometry and function theory on total spaces of vector bundles on an elliptic curve (or a complex torus). In particular, our investigation is guided by the following question.

\begin{question}\label{queintro}

\begin{enumerate}[label=(\arabic*)]
    \item   On a compact Riemann surface $M$, Weil's result associates a representation $\pi_1(M) \rightarrow \operatorname{GL}(n, \mathbb{C})$ with a vector bundle of degree zero on $M$. However, this correspondence is not unique, see \cite[p.~128]{Newstead} and {\cite[p.~578]{Donaldson2021}}. Can we characterize those representations which correspond to isomorphic vector bundles?

   \item Study the fine structure of the space of holomorphic functions on the total space of a vector bundle on an elliptic curve. Is it possible that the total spaces of two non-isomorphic vector bundles are biholomorphic complex manifolds?

   \item  What is the `best' complete K\"ahler metric on the total space of a given vector bundle over an elliptic curve? Does the existence of a K\"ahler metric satisfying suitable curvature conditions on the total space relate to the algebraic structure of the vector bundle? 
\end{enumerate}
\end{question}

We remark that Questions (1) and (2) have been studied in various aspects. For example, the pull-back bundle over $\mathbb{C}$ of any vector bundle over an elliptic curve is trivial. Then there exists a correspondence between isomorphism classes of vector bundles and equivalence classes of factors of automorphy. We refer to \cite{Iena} for further classification results along this line. For a vector bundle which is negative in the sense of Grauert \cite{Grauert1962}, the zero section can be recovered intrinsically from the total space as its maximal positive-dimensional compact analytic locus. Therefore any biholomorphism between two such total spaces must map the zero section to the other zero section. Consequently, the two bundles are isomorphic after taking normal bundles of zero sections. We will return to these questions in later sections.

\subsection{Statement of main results} 

Let $\Sigma$ be the elliptic curve defined in (\ref{ecurve}). According to Appell-Humbert's theorem (see, for example, \cite[Theorem, p.~19]{Mumford} and \cite[Theorem~5.17]{Debarre}), any line bundle over $\Sigma$ is isomorphic to some $L(H, \alpha)$. Here $H$ is a Hermitian bilinear form on $\mathbb{C}$ so that $\operatorname{Im}H$ is $\mathbb{Z}$-valued when restricted onto $\Gamma \times \Gamma$, and $\alpha: \Gamma \rightarrow S^1$ satisfies (\ref{alphasemi}). The line bundle $L(H, \alpha)$ over $\Sigma$ is defined as a quotient of $\mathbb{C}^2$ under the action
\begin{align}
(z, \xi) \rightarrow (z+\gamma, \alpha(\gamma) e^{\pi H z \overline{\gamma}+\frac{\pi}{2}H |\gamma|^2} \xi).    \label{AHaction_intro}
\end{align} It is direct to check $\operatorname{deg}(L(H, \alpha))=H \operatorname{Im}\tau$.

Motivated by the indecomposable vector bundle defined in (\ref{Edef}) studied by \cite{Atiyah} and \cite{DPS94}, we consider a vector bundle $E$ over $\Sigma$ which is associated with a representation $\rho: \pi_1(\Sigma) \rightarrow \operatorname{GL}(2, \mathbb{C})$
\begin{equation}
    \rho(\gamma_1)=e^{2\pi i\theta_1} \begin{pmatrix}
        1 & b_1\\
        0& 1
    \end{pmatrix}, \ \ 
    \rho(\gamma_2)=e^{2\pi i\theta_2} \begin{pmatrix}
       1 & b_2\\
        0 & 1
    \end{pmatrix}.   \label{gen-rep}
\end{equation} Recall that $\pi_1(\Sigma)$ is generated by $\gamma_1$ and $\gamma_2$ in (\ref{funda_group}). Here we assume $\theta_i \in [0, 1)$ and $b_i \in \mathbb{C}$ for $i\in \{1, 2\}$. 

We consider a rank $2$ vector bundle $E$ of degree zero over $\Sigma$ which belongs to one of the following types.
\begin{itemize}
    \item  \namedlabel{T1}{\textbf{Type I}}: $E=L_1 \oplus L_2$ where $L_1, L_2 \in \operatorname{Pic}^0(\Sigma)$.

    \item  \namedlabel{T2}{\textbf{Type II}}: $E=L_1 \oplus L_2$ where $\operatorname{deg}L_1=-\operatorname{deg}L_2>0$.

    \item  \namedlabel{T3}{\textbf{Type III}}: $E$ is indecomposable and associated with a representation defined in (\ref{gen-rep}).
\end{itemize}

We observe that any rank $2$ vector bundle of degree zero over $\Sigma$ belongs to exactly one of \ref{T1}, \ref{T2}, and \ref{T3}. It follows from \cite{Weil} that any indecomposable rank $2$ bundle $E$ of degree $0$ must be associated with some representation $\rho$ of $\pi_1(\Sigma)$ into $\operatorname{GL}(2, \mathbb{C})$. By a standard fact in linear algebra, two commuting matrices $\rho(\gamma_1)$ and $\rho(\gamma_2)$ are simultaneously similar to upper triangular matrices. Since $E$ is indecomposable, we may further assume that at least one of $\rho(\gamma_1)$ and $\rho(\gamma_2)$ has multiple eigenvalues. In fact, by \cite{Atiyah}, any such indecomposable bundle $E$ admits an extension of some $L \in \operatorname{Pic}^{0} (\Sigma)$ by itself. Therefore, any indecomposable rank $2$  bundle $E$ of degree $0$ corresponds to some upper triangular representation $\rho: \pi_1(\Sigma) \rightarrow \operatorname{GL}(2, \mathbb{C})$ defined in (\ref{gen-rep}). Except for this aspect of Atiyah's result, our main results study vector bundles of \ref{T1}, \ref{T2}, and \ref{T3} directly and do not rely on the full power of the classification results in \cite{Atiyah}.

Let $E$ be a rank $2$ vector bundle of degree $0$ over $\Sigma$ and $X$ its total space. Let $\operatorname{Aut}(M)$ denote the group of biholomorphisms of a complex manifold $M$. By a delicate analysis of $\operatorname{Aut}(\mathbb{C}^3)$ under an invariant condition related to \eqref{AHaction_intro} and  \eqref{gen-rep}, we can show some rigidity results on $\mathcal{O}(X)$ and $\operatorname{Aut}(X)$; see Propositions \ref{Xfunction} and \ref{biholoC} for precise statements.

\begin{enumerate}[label=(\arabic*)]
        
\item  $E$ is of \ref{T3} if and only if it is associated to a representation in (\ref{gen-rep}) with $b_2 \neq b_1 \tau$. In particular, any two vector bundles $E$ and $\widetilde{E}$ of \ref{T3} are isomorphic if the corresponding $\theta_k=\widetilde{\theta}_k$ for $k=1, 2$ in (\ref{gen-rep}).

\item  If $E$ belongs to any of \ref{T1}, \ref{T2}, and \ref{T3}, we have a complete characterization of $\mathcal{O}(X)$, see (\ref{taylor1}), (\ref{taylor2}), and (\ref{taylor3}) for a detailed statement. In particular, if $E$ is of \ref{T3}, then $\mathcal{O}(X)$ contains a nonconstant element if and only if both $\theta_1$ and $\theta_2$ in (\ref{gen-rep}) are rational. Note that $X$ is not holomorphically convex in this case.

\item  Let $\Sigma$ (resp.~$\widetilde{\Sigma}$) be the elliptic curve defined by
the lattice generated by $\{1,\tau\}$ (resp.~$\{1,\widetilde{\tau}\}$), as
in \eqref{ecurve}. Let $X$ (resp.~$Y$) be the total space of a rank-$2$
vector bundle $E$ (resp.~$\widetilde E$) of degree zero over
$\Sigma$ (resp.~$\widetilde{\Sigma}$). Assume there exists a biholomorphic map $F: X \rightarrow Y$. Then there exists a biholomorphism $f: \Sigma \rightarrow \widetilde{\Sigma}$ and a corresponding bundle isomorphism $\Phi: X \rightarrow Y$. Moreover, any such bundle isomorphism $\Phi$ can be constructed explicitly. In the case of \ref{T3}, we can determine all biholomorphisms between the total spaces.
\end{enumerate}

Note that in part (3), by a bundle isomorphism map we allow the usual definition of a bundle isomorphism (see, for example, \cite[Definition 5.2, p.~62]{MK2006}) to cover an isomorphism of the base spaces. Precisely, let $\pi_E\colon E \to \Sigma$ and $\pi_{\widetilde E} \colon \widetilde E \to \widetilde{\Sigma}$ be
holomorphic vector bundles, and $X$ and $Y$ be corresponding total spaces. A biholomorphism
\(
\Phi \colon X \rightarrow Y
\)
is called a \emph{bundle isomorphism} if there exists a biholomorphism $f \colon \Sigma \to \widetilde{\Sigma}$ such that
\(
\pi_{\widetilde E}\circ\Phi=f\circ\pi_E
\)
and, for every $p\in\Sigma$, the restriction
\(
\Phi|_{E_p}\colon E_p\longrightarrow {\widetilde E}_{f(p)}
\) is complex linear.

Our next goal is to construct a graded linear structure on a subset in $\mathcal{O}(X)$. From the function-theoretic point of view, we say that any given $f \in \mathcal{O}(X)$ is of \emph{polynomial type} if the corresponding Taylor series in one of (\ref{taylor1}), (\ref{taylor2}), and (\ref{taylor3}) terminates after finite terms. One may check that this definition is independent of the choice of local coordinates in (\ref{taylor1}), or the choice of lifts in (\ref{taylor2}), and (\ref{taylor3}). In the meantime, we may introduce a complete Hermitian metric $g$ on $X$ and a corresponding subspace $\mathcal{O}_d(X, g) \subset \mathcal{O}(X)$. Namely, let $p$ be a fixed point on a complete Hermitian manifold $(X, g)$, and $d_{g}$ the distance function with respect to $g$. For any real number $d>0$, we define
\begin{equation}
   \mathcal{O}_d(X, g)=\{f \in \mathcal{O}(X)\ |\ \exists\ C(d, f)>0,\  |f(q)| \leq C(d_g(q, p)+1)^{d}\ \text{for all}\ q \in X\}.   \label{polyspaced}
\end{equation}
We construct special Hermitian metrics on $X$ such that the corresponding $\mathcal{O}_d(X, g)$ meshes well with holomorphic functions of polynomial type.

\begin{theorem}\label{Hermipoly}
Let $X$ be the total space of any rank $2$ bundle $E$ of degree zero over an elliptic curve $\Sigma$ defined in (\ref{ecurve}). There exists a complete Hermitian metric $g$ on $X$ which has zero first Chern-Ricci curvature and is Gauduchon in the sense that $\sqrt{-1}\partial \bar{\partial} (\omega^2)=0$. If $E$ is of \ref{T3}, we get a family of such Hermitian metrics which are quasi-isometric to each other. Moreover, $\mathcal{O}_d(X, g)$ defined in (\ref{polyspaced}) consists of all $f \in \mathcal{O}(X)$ whose corresponding power series in one of (\ref{taylor1}), (\ref{taylor2}), and (\ref{taylor3}) are finite and of highest degree $\leq \lfloor d \rfloor$. Here $\lfloor d \rfloor$ denotes the greatest integer $\leq d$.
\end{theorem}

By Calabi's work \cite{Calabi1}, the total space of any Hermitian vector bundle over a compact K\"ahler manifold admits a complete K\"ahler metric. Let $E$ be the indecomposable rank $2$ bundle defined in (\ref{Edef}). There is a natural choice of Hermitian (fiber) metric on $E$ introduced by Paun {\cite[p.~53]{Paun2001}}. In Proposition \ref{holoCala}, we show that any nonconstant holomorphic function on the total space $X$ grows with the Hadamard order $\geq 2$ with respect to any corresponding K\"ahler metric of Calabi type on $X$. Therefore, Theorem \ref{Hermipoly} shows that function theory in Hermitian geometry might be more `natural' than that in K\"ahler geometry in some cases. In the meantime, Hermitian metrics constructed in Theorem \ref{Hermipoly} could be viewed as non-K\"ahler Calabi-Yau metrics on a natural class of noncompact manifolds. We refer the reader to \cite{Tosatti} for a nice introduction on compact non-K\"ahler Calabi-Yau manifolds.

\begin{remark}\label{higher_rank}
We have focused on the rank $2$ bundles with degree zero defined by \eqref{gen-rep} so far. These representations are special cases of factors of automorphy studied in \cite{Iena}. We refer to \cite[Section 5.1]{Iena} for the representation $\pi_1(\Sigma) \rightarrow \operatorname{GL}(r, \mathbb{C})$ which corresponds to a particular indecomposable rank $r$ bundle of degree zero in \cite{Atiyah}. This bundle is a higher rank analogue of the rank $2$ bundle defined in \eqref{Edef}. One may expect that the conclusion of Theorem \ref{Hermipoly} holds for any degree zero bundle over an elliptic curve.
\end{remark}

Our main result is a complete characterization of complete K\"ahler metrics with nonnegative holomorphic bisectional curvature ($BI \geq 0$ for short) on the total space of a vector bundle over an elliptic curve.

\begin{theorem}\label{res_3_intro}

(1)   Let $E$ be a rank $r$ vector bundle over an elliptic curve $\Sigma$ defined in (\ref{ecurve}). Assume that the total space $X$ admits a complete K\"ahler metric $\omega$ with $BI \geq 0$. Then $E=L_1 \oplus \cdots \oplus L_r$ where each $L_i$ is a line bundle of $\deg L_i=0$. Moreover, $X$ admits a complete flat K\"ahler metric.

(2)   Let $n\geq 2$ be an integer and $L(H,\alpha)$ be a line bundle over a compact complex torus $M$ of complex dimension $n-1$; see (\ref{Hform1}) to (\ref{AHaction}) for a definition of $L(H,\alpha)$ and $M$. Assume that its total space $X$ admits a complete K\"ahler metric $g$ with $BI \geq 0$ so that the lifting metric $\widetilde{g}$ on the universal covering $\widetilde{X}$ splits as $\widetilde{g}=g_e+g_0$. Here $g_e$ is the standard Euclidean metric on $\mathbb{C}^{n-1}$, and $g_0$ is a complete K\"ahler metric with nonnegative Gauss curvature on $\mathbb{C}$. Then $H=0$.

(3)  Let $X$ be the total space of the same line bundle in part (2). Then any complete K\"ahler metric $\omega$ with $BI \geq 0$ (whose lifted metric satisfies the same assumption in part (2)) admits an explicit characterization; see \eqref{gexpress2} and \eqref{geexpress4}. In particular, if $M$ is the elliptic curve $\Sigma$ in (\ref{ecurve}), let $\alpha(1)=e^{2\pi i \theta_1}$ and $\alpha(\tau)=e^{2\pi i \theta_2}$ for some $\theta_1, \theta_2 \in [0, 1)$. Then any complete non-flat K\"ahler metric $\omega$ can be completely determined. 
       \begin{enumerate} [label=(\roman*)]
       \item If either $\theta_1$ or $\theta_2$ is irrational, then there exist
       \begin{itemize}
           \item  a constant $C>0$.
           \item  a smooth positive function $\widetilde{u}(\xi)$ on $\mathbb{C}$ so that $\ln \widetilde{u}$ is superharmonic and $\widetilde{u}(\xi)|d\xi|^2$ is complete on $\mathbb{C}$.
           \item  $\widetilde{u}(\xi)$ is rotationally symmetric as $\widetilde{u}(\xi)=\widetilde{u}(|\xi|)$ for any $\xi \in \mathbb{C}$
       \end{itemize}
       such that
     \begin{equation}
    \omega=\sqrt{-1}( C dz \wedge d\overline{z} +\widetilde{u}(|\xi|) d\xi \wedge d \overline{\xi}).  \label{ir_metric_intro}
     \end{equation} 
    \item  If both $\theta_1$ and $\theta_2$ are rational, let $k$ denote the minimal positive integer so that $k\theta_1$ and $k\theta_2$ are integers. Then there exist 
    \begin{itemize}
         \item  a constant $C>0$.
        \item  a holomorphic function $h(\xi)=\widehat{h}(\xi^k)$ for an entire function $\widehat{h}$ on $\mathbb{C}$.
        \item  a smooth positive function $\widetilde{u}(\xi)$ on $\mathbb{C}$ so that $\ln \widetilde{u}$ is superharmonic and $\widetilde{u}(\xi)|d\xi|^2$ is complete on $\mathbb{C}$.
        \item  $\widetilde{u}(\xi)$ is $\mathbb{Z}_k$-invariant in the sense that $\widetilde{u}(\xi)=\widetilde{u}(e^{\frac{2\pi i}{k}}\xi)$ for any $\xi \in \mathbb{C}$
    \end{itemize}
    such that
   \begin{equation}
    \omega=\sqrt{-1}( C d(z+ h(\xi)) \wedge d(\overline{z}+ \overline{h(\xi)}) +\widetilde{u}(\xi) d\xi \wedge d \overline{\xi} ).    \label{r_metric_intro}
   \end{equation}
      \end{enumerate} 
      Here in both (\ref{ir_metric_intro}) and (\ref{r_metric_intro}), $(z, \xi)$ denote the corresponding (local) coordinates on $X$ for those on $\mathbb{C}^2$ defined in (\ref{AHaction_intro}).
   Moreover, in both flat and non-flat case ((\ref{ir_metric_intro}) and (\ref{r_metric_intro})), we can find some $F \in \operatorname{Aut}(X)$ which realizes $X$ as a total space of an isomorphic bundle $\widehat{L}$ over an elliptic curve $\widehat{\Sigma}$. Moreover, $F^{\ast}\omega$ is a product metric on $\widehat{\pi}^{-1}(U)$ where $\widehat{\pi}: \widehat{L} \rightarrow \widehat{\Sigma}$ is the projection map and $U$ is a local trivialization on $\widehat{\Sigma}$.

\end{theorem}

\begin{remark}
In part (2) and (3) of Theorem \ref{res_3_intro}, we assume that the universal covering of $(X, g)$ splits with a complex Euclidean factor of dimension $n-1$. In the first version posted on arXiv (arXiv:2511.08906v1), we show that this assumption is redundant in the case $n=2$ using a result of Ni--Tam \cite{NT2003}, and suspect that the same might hold for any $n \geq 3$. The intuition is that the rank of the fundamental group of the compact torus $M$ might contribute to the dimension of flat Euclidean factor on $\widetilde{X}$. For example, if we assume ``nonnegative Riemann sectional curvature" instead of ``$BI \geq 0$", the universal covering splits a flat factor $\mathbb{C}^{n-1}$ by \cite[Corollary 5.1]{NT2003}. It turns out that this assumption is indeed redundant in view of a later result \cite[Theorem~1.2]{CMS2026} (arXiv:2601.14231).
\end{remark}

It is natural to consider Question \ref{queintro} for vector bundles with nonzero degrees over an elliptic curve. What about function theory on the total spaces? A natural approach is to explore the method of factors of automorphy studied in \cite{Iena}. In Section \ref{sec6}, we pursue the method of Fuchsian representation of an orbifold fundamental group which originated in \cite{NS1965}. In Proposition \ref{n2d-1noholo}, we give an explicit representation of some basis elements of $\mathcal{O}(X)$ by studying automorphic meromorphic functions and sections on specific elliptic curves. Applications include the holomorphic convexity of $X$ in a special case. Are there any interesting complete Kähler metrics on such $X$? We know little beyond the $BI \geq 0$ case considered in Theorem \ref{res_3_intro} and some results for the $BI \leq 0$ case obtained in \cite{WY2025}.

\subsection{Further discussions}

For any fixed K\"ahler manifold $(M, g)$, we define its Ricci rank as the maximum rank of the Ricci curvature over $M$. Wu and Zheng ask the following question in \cite[Conjecture on p.~307]{WZ2003}: Let $(M^n, g)$ be a complete K\"ahler manifold with $BI \geq 0$ and its Ricci rank $n-k$. Then its universal covering $(\widetilde{M}, \widetilde{g})$ is holomorphically isometric to $(N^{n-k} \times \mathbb{C}^k, h+g_{e})$ where $h$ is a complete K\"ahler metric on $N$ with $BI \geq 0$, and $g_e$ is the Euclidean metric on $\mathbb{C}^k$. If the curvature of $(M, g)$ is bounded, a stronger structure theorem on its universal cover was obtained by Cao--Chen--Zhu \cite[Theorem 2.8]{CCZ2008} using the K\"ahler-Ricci flow.

Motivated by structural theorems of K\"ahler manifolds with $BI \geq 0$ due to Howard--Smyth--Wu \cite{HSW1981}, Mok \cite{Mok1988}, Takayama \cite[Theorem 1.6]{Taka1998}, Wu--Zheng \cite{WZ2003}, Zheng \cite[Corollary 5.1]{NT2003}, Ni--Tam \cite[Theorem 5.1]{NT2003}, Cao--Chen--Zhu \cite{CCZ2008}, and Chau--Tam \cite[Corollary 3.1]{CT2011}, we propose the following question. 

\begin{question}\label{Biopen}
    Let $(M^n, g)$ be a complete K\"ahler manifold with $BI(g) \geq 0$ and its Ricci rank be $n-k$ where $1 \leq k \leq n$. Then there exists a finite covering $(\widehat{M}, \widehat{g})$ with a locally trivial holomorphic fibration $\rho: \widehat{M} \rightarrow M_1$ where $M_1$ is a complete flat K\"ahler manifold of complex dimension $k$. Moreover, each fiber of $\rho$ with the induced metric from $\widehat{g}$ is a simply-connected holomorphically isometric product manifold $M_2$ with each factor being one of the following: 
       \begin{enumerate} [label=(\arabic*)]
        \item  A simply connected complete noncompact K\"ahler manifold $(N_1, g_1)$ so that $N_1$ admits another complete K\"ahler metric with $BI>0$. According to Yau's uniformization conjecture \cite{Yau1991}, $N_1$ is expected to be biholomorphic to a complex Euclidean space. 
        
       \item A product of complex projective spaces each of which is endowed with a K\"ahler metric with $BI \geq 0$. 
        
        \item A product of compact Hermitian symmetric spaces of rank $\geq 2$ each of which is endowed with its canonical K\"ahler-Einstein metric.
    \end{enumerate}
\end{question} 

\begin{remark}
In view of part (3) of Theorem \ref{res_3_intro}, we may expect that up to some biholomorphism $F$ of $\widehat{M}$, $\widehat{g}$ is a product metric on $\rho^{-1}(U)$ for a local trivialization $U$ on $M_1$. Consider the compact case studied in \cite{HSW1981} and \cite{Mok1988} where only (2) and (3) in Question \ref{Biopen} could appear. If we fix a fiber bundle structure on $\widehat{M}$, then its fiber is totally geodesic with respect to the induced metric of $F^{\ast}\widehat{g}$ for any choice of $F \in \operatorname{Aut}(\widehat{M})$. However, in Question \ref{Biopen}, for a fixed fiber bundle structure on $\widehat{M}$, the factor $(N_1, g_1)$ in (1) of Question \ref{Biopen} is not totally geodesic for a general choice of $F \in \operatorname{Aut}(\widehat{M})$. In Example \ref{negexam}, we construct a line bundle over an elliptic curve whose total space $X$ admits a complete K\"ahler metric with $BI \geq 0$ such that the induced metric along the fiber has nonpositive curvature.
\end{remark}

The complete flat K\"ahler manifold $M_1$ in Question \ref{Biopen} can be written as $\mathbb{C}^{k} / G$ where $G$ is a fixed-point-free discrete subgroup of the holomorphic isometry group of $\mathbb{C}^{k}$. When $G$ is not trivial, we have the following three cases. Recall that a finite cover has been taken into account in Question \ref{Biopen}.
\begin{enumerate} [label=(\roman*)]
       \item If $G$ is a lattice of real rank $2k$, then $M_1$ is a compact complex torus $T_{\mathbb{C}}^{k}$.

       \item If $G$ is a lattice of real rank $1 \leq l <2k$, by a theorem of Remmert and Morimoto \cite[Theorem~3.2]{Morimoto1966}, $M_1$ is holomorphically isomorphic to a product of $\mathbb{C}^{l_1}$, $(\mathbb{C}^{\ast})^{l_2}$, and a toroidal group. Recall a toroidal group is an abelian complex Lie group on which every holomorphic function is constant. In general, any toroidal group is a holomorphic fiber bundle onto a complex torus with its fiber being a product of copies of $\mathbb{C}^{\ast}$, see \cite{YK2001} for a general introduction of toroidal groups. Cousin \cite{Cousin} constructed the first example of toroidal groups. For example, let $\Gamma$ be a lattice in $\mathbb{C}^2$ generated by $(1, 0), (0, 1)$, and $(i, ia)$ where $a$ is irrational. We may check by power series that $\mathbb{C}^2/\Gamma$ is a toroidal group; see {\cite[p.~262]{Morimoto1965}} for further constructions of toroidal groups. 

       \item  If $G$ is a semi-direct product associated with a lattice $\Gamma$ and some representation $\rho: \Gamma \rightarrow U(r)$, $M_1$ is related to a holomorphic fiber bundle over $\mathbb{C}^{k-r}/\Gamma$ with fiber being a product of $\mathbb{C}^{r_1}$ and $(\mathbb{C}^{\ast})^{r_2}$ for two integers $r_1$ and $r_2$ with $r_1+r_2=r$.
\end{enumerate}
     In view of Theorem \ref{res_3_intro}, it seems interesting to study total spaces of vector bundles over a complete flat K\"ahler manifold which appears in the above list. We expect that the existence of a complete K\"ahler metric with $BI \geq 0$ on total spaces imposes strong restrictions on the algebraic properties of the corresponding vector bundle.

\begin{comment}
\begin{proposition}\label{rankkflat}
    Let $E=\oplus_{p=1}^k L_p(H_p,\alpha_p)$ where each $L_p$ is a line bundle over a compact complex torus $M=\mathbb{C}^{n-k} / \Gamma$. Let $X$ be the total space of $E$. Then $X$ admits a complete flat K\"ahler metric if and only if $H_p=0$ for $p=1,\cdots,k$. In particular, if $k=1$, there exists a characterization (see (\ref{gexpress3})) of any complete flat K\"ahler metrics on $X$ in terms of $(z, \xi)$ which are the corresponding coordinates descended from those on $\mathbb{C}^2$ defined in (\ref{AHaction_intro}).
\end{proposition}    
\end{comment}

There has been important progress on the uniformization of complete K\"ahler manifolds with $BI \geq 0$ in recent years. We refer the reader to \cite{CT2006, CT2008, CZ2018, Liu2019, LT2020, Lott2021} for further details.

The structure of the paper is the following. In Section \ref{sec2} we apply Calabi's construction to study the growth of holomorphic functions on the total space of a vector bundle, proving Proposition \ref{holoCala}. In Section \ref{sec3} we study the K\"ahler geometry of total spaces of line bundles over a complex torus and prove Theorem \ref{res_3_intro}. In Sections \ref{sec4} and \ref{sec5} we study the function theory and Hermitian geometry on total spaces of rank $2$ bundles with zero degree, proving Theorem \ref{Hermipoly}. Finally, we discuss the function theory on total spaces of certain bundles with nonzero degree in Section \ref{sec6}.

\vskip 0.2cm

\noindent\textbf{Acknowledgments.} Bo Yang would like to thank Fangyang Zheng for his interest, Fusheng Deng, Zhan Li, and Wenfei Liu for helpful discussions.

\begin{comment}
    This work was partially supported by National Natural Science Foundation of China [11801475, 12141101, 12271451] and Natural Science Foundation of Fujian Province of China [2019J05012]. 
\end{comment}

\section{Function theory via the Calabi method}\label{sec2}

In this section, we review Calabi's theorem that the total space of any holomorphic vector bundle over a compact K\"ahler manifold admits a complete K\"ahler metric, and apply it to study function theory on $X$ which is the total space of the vector bundle $E$ defined in (\ref{Edef}). The main result is Proposition \ref{holoCala} which gives a sharp estimate on the growth of holomorphic functions on $X$.

Let $X$ be the total space of $\pi: (E, h) \rightarrow M$ which is a Hermitian vector bundle over a complete K\"ahler manifold $(M, g)$. Calabi considered the K\"ahler metric $\widetilde{g}$ on $X$ in the form of
\begin{align}
\widetilde{\omega}=\pi^{\ast}\omega_g+\sqrt{-1}\partial\overline{\partial} u.  \label{Cmetric}
\end{align}
Here for any point $(p, v^{\mu} e_{\mu}) \in X$, $u$ is a smooth single-variable function on $t=h(v^{\mu} e_{\mu}, \overline{v^{\mu} e_{\mu}})$ which is to be determined. For simplicity, we say that a K\"ahler metric is \emph{of the Calabi type} if it is in the form of (\ref{Cmetric}).

\begin{theorem}[Calabi {\cite[Theorem 3.2 on p.~275]{Calabi1}}]\label{vecmetric}
Let $X$ be the total space of $E$ which is a holomorphic vector bundle over a complete K\"ahler manifold $(M, g)$. Assume $E$ admits a Hermitian metric
$h$ so that its Chern curvature is bounded from above, i.e. there exists some constant $C_1$ so that
\begin{align}
R(U, \overline{U}, v^{\alpha} e_{\alpha}, \overline{v^{\alpha} e_{\alpha}}) \leq  C_1 g(U, \overline{U}) h(v^{\alpha} e_{\alpha}, \overline{v^{\alpha} e_{\alpha}})  \label{ubound}
\end{align}
for any $p \in M$, $U \in T_{p}^{1, 0}M$, and $v^{\alpha} e_{\alpha} \in \pi^{-1}(p)$. Then $X$ admits a complete K\"ahler metric of the Calabi type.
\end{theorem}

It is also mentioned in {\cite[p.~272]{Calabi1}} that any K\"ahler metric in the form of (\ref{Cmetric}) satisfies that each fiber is totally geodesic. We also refer to \cite[Proposition 2.2 on p.~2297]{HS2002} for a proof. For the sake of completeness, we review this fact in the special case of a rank $2$ bundle over an elliptic curve $\Sigma$. We also discuss the choice of the function $u$ in (\ref{Cmetric}). These two ingredients are crucial in Calabi's proof of Theorem \ref{vecmetric}.

For any $p \in \Sigma$, choose local holomorphic normal coordinates  $(z, v^1, v^2)$
centered at $p$ such that $z(p)=0$ and $\partial_z h(e_{\alpha},\overline {e_{\beta}})=0$ at $p$.
Recall $t=h(v, \overline v)=v^{\alpha}\overline {v^{\beta}}h(e_{\alpha},\overline{e_{\beta}})$.
\begin{equation}
     \partial \bar{\partial} u(t)=u^{\prime}(t)\partial \bar{\partial} t +u^{\prime\prime}(t)\partial t \wedge \bar\partial t.
\end{equation}
Hence on $\pi^{-1}(p)$, the local components of $\widetilde{g}$ are
\begin{equation}
    \widetilde{g}=\begin{pmatrix}
        g_{z \bar z}+u'(t)t_{z\bar{z}} &0 &0 \\
        0 & u'+u''|v^1|^2 &u'' \overline{v^1}  v^2 \\
        0&u'' v^1\overline{v^2} & u'+u''|v^2|^2
    \end{pmatrix}.
\end{equation}
It follows that $\Gamma_{\alpha \beta}^z =\widetilde{g}^{z \bar l} \partial_{\alpha} \widetilde{g}_{\beta \bar l}=0$ holds along $\pi^{-1}(p)$. Therefore, the fiber $\pi^{-1}(p)$ is totally geodesic. 

Next, we consider the growth of geodesics along a fixed fiber. Let $\widetilde{\omega}_p =\widetilde{\omega}|_{\pi^{-1}(p)}$ be the induced metric along $\pi^{-1}(p)$. Then for $v=(v^1,0) \in \pi^{-1}(p)$, we may estimate the distance of $v$ to the zero section $\Sigma$ by
\begin{equation}
     d_{\widetilde{\omega}_p}(v,0)=\int_0^{|v|^2}\sqrt{u'(t)+tu''(t)}\frac{1}{2\sqrt t}\,dt.
    \label{fiberdist}
\end{equation}
Therefore, any positive-definite $(1,1)-$form $\widetilde{\omega}$ in the form of \eqref{Cmetric} is a complete K\"ahler metric if $u'(t)+tu''(t)>0$ and the integral in \eqref{fiberdist} is divergent. 
In \cite{Calabi1}, the following form of $u(t)$ is considered
\begin{equation}
    u(t)=A\ln(C+t)-B\ln \ln (C+t),   \label{uchoice}
\end{equation}
where $C>0$ is a large constant and $A,B$ are constants to be determined. 
\begin{align*}
 &u'(t)=\frac{1}{C+t}(A-\frac{B}{\ln (C+t)}),  \\
 &u''(t)=-\frac{1}{(C+t)^2}(A-\frac{B}{\ln (C+t)}-\frac{B}{(\ln (C+t))^2}), \\
 &u'(t)+tu''(t)=\frac{1}{(C+t)^2}(\frac{Bt}{(\ln(C+t))^2}+AC-\frac{BC}{\ln (C+t)}). 
\end{align*}   
Let $A> \frac{B}{\ln C}>0$. Then $u'(t)>0$ for $t>0$. Moreover, we have
\begin{equation}
    \sqrt{u'(t)+tu''(t)}\frac{1}{2\sqrt t}\ge \frac{\sqrt{B}}{2(C+t)\ln(C+t)}.
\end{equation}
In view of \eqref{fiberdist}, we conclude that the distance function along each fiber goes to infinity. To sum up, we choose $A$ satisfying $C_1 A <1$ and then pick $B>0,C>1$ such that $A> \frac{B}{\ln C}>0$. It follows that the K\"ahler metric in the form of \eqref{Cmetric} with $u$ of the form \eqref{uchoice} is complete.

As an application of the Calabi method, we study the growth of holomorphic functions of the vector bundle $E$ defined in (\ref{Edef}). For simplicity, we assume $\tau=\sqrt{-1}$ in (\ref{Edef}), and consider the following Hermitian metric on the trivial rank $2$ bundle over $\mathbb{C}$, 
\begin{align}
    |(v^1, v^2)|_h^2=|v^1-(\operatorname{Im}z)v^2|^2+k|v^2|^2\ \text{for any}\ k>1.  \label{Hdef}
\end{align} 
Such a family of Hermitian (fiber) metrics was introduced by Paun {\cite[p.~53]{Paun2001}}. Note that (\ref{Hdef}) is invariant under the action of $(\Gamma, \rho)$ defined in (\ref{Edef}). Let $h$ denote the corresponding Hermitian metric on $E$ over the elliptic curve $\Sigma=\mathbb{C}/\Gamma$. According to the Calabi method, we may study K\"ahler metrics on $X$ in the form of (\ref{Cmetric}) where $\omega_0=\sqrt{-1}dz \wedge d\overline{z}$ is a flat metric on $\Sigma$. Let $z=x+\sqrt{-1} y$ and $t=h(v, \overline{v})=|v^1-yv^2|^2+k|v^2|^2$. Then:
\begin{align*}
    &h_{\alpha \bar \beta}=\begin{pmatrix}
        1 &-y 
        \\-y & y^2+k
    \end{pmatrix},\ \ \ 
    h^{\alpha \bar \beta}=\frac{1}{k}\begin{pmatrix}
        y^2+k &y 
        \\y & 1
    \end{pmatrix}.\\
    &\partial_z h_{\alpha \bar \beta}=\begin{pmatrix}
        0 &\frac{\sqrt{-1}}{2} 
        \\ \frac{\sqrt{-1}}{2} & \frac{y}{\sqrt{-1}}
    \end{pmatrix},\  \ 
    \Gamma_{\alpha z}^{\beta}=h^{\beta \bar \delta} \partial_z h_{\alpha \bar \delta}=\frac{\sqrt{-1}}{2k}\begin{pmatrix}
        y & 1 \\
        k-y^2 & -y
    \end{pmatrix}.\\
    & \partial_{\bar z}\Gamma_{z \alpha}^{\beta}=-\frac{1}{4k}\begin{pmatrix}
        1 & 0 \\
        -2y & -1
    \end{pmatrix},\ \ 
    R^E_{z \bar z \alpha \bar \beta}=-\partial_{\bar z}\Gamma_{\alpha z}^{\gamma} h_{\gamma \bar \beta}=\frac{1}{4k}\begin{pmatrix}
        1 & -y \\
        -y & y^2-k
    \end{pmatrix}.
\end{align*}
According to \cite[p.~274]{Calabi1}, we may rewrite (\ref{Cmetric}) as
\begin{align}
\omega_g=\sqrt{-1}\Big( (1-u'(t)R^E_{z \bar z \alpha \bar \beta}v^{\alpha}\overline{v^{\beta}}) dz \wedge d\bar z+(u''(t) h_{\alpha \bar{\xi}} \overline{v^{\xi}} h_{\eta \bar\beta}v^{\eta}+u'(t) h_{\alpha \bar\beta})\nabla v^{\alpha} \wedge \nabla \overline{v^{\beta}}\Big)  \label{Cmetric2}
\end{align}
where
\begin{equation}
    \nabla v^{\alpha}=dv^{\alpha} +(v^1 \Gamma_{1z}^{\alpha} +v^2 \Gamma_{2z}^{\alpha})dz.
\end{equation}
With the choice of Hermitian metric $h$ on $E$ in the form of (\ref{Hdef}), we may solve
\begin{align*}
    & \nabla v^1=dv^1+(v^1\Gamma_{1z}^1+v^2\Gamma_{2z}^1)dz=dv^1+\frac{\sqrt{-1}}{2k}(yv^1+(k-y^2)v^2)dz.\\
    & \nabla v^2=dv^2+(v^1\Gamma_{1z}^2+v^2\Gamma_{2z}^2)dz=dv^2+\frac{\sqrt{-1}}{2k}(v^1-yv^2)dz.
\end{align*}
\begin{align*}
    & R^E_{z \bar z \alpha \bar \beta}v^{\alpha} \overline{v^{\beta}}=\frac{1}{4k}(|v^1|^2-y(v^1 \overline{v^2}+v^2 \overline{v^1})+(y^2-k)|v^2|^2)=\frac{1}{4k}(|v^1-yv^2|^2-k|v^2|^2).\\
    & h_{\alpha \bar{\xi}} \overline{v^{\xi}} h_{\eta \bar\beta} v^{\eta}=
    \begin{pmatrix}
        |v^1-yv^2|^2 & (\overline{v^1}-y\overline{v^2})(-yv^1+(y^2+k)v^2)
        \\(v^1-yv^2)(-y\overline{v^1}+(y^2+k)\overline{v^2}) & |-yv^1+(y^2+k)v^2|^2
    \end{pmatrix}.
\end{align*}
If we introduce
\begin{equation*}
    \psi^1=\nabla v^1-y\nabla v^2, \ \psi^2=\nabla v^2, \ a=v^1-yv^2, \ b=kv^2.
\end{equation*} and
\begin{equation*}
    A_{\alpha \bar \beta}=\begin{pmatrix}
        |a|^2 &b\bar a
        \\ a\bar b& |b|^2
    \end{pmatrix},\ \ 
    B_{\alpha \bar\beta}=\begin{pmatrix}
        1&0\\0 &k
    \end{pmatrix}.
\end{equation*}
Then (\ref{Cmetric2}) can be further reduced to
\begin{equation}
    \omega_g=\sqrt{-1}\Big[ \Big(1-\frac{1}{4k}u'(t) (|v^1-yv^2|^2-k|v^2|^2)\Big) dz \wedge d\bar z+(u''A_{\alpha \bar \beta}+u'B_{\alpha \bar\beta}) \psi^{\alpha} \wedge \overline{\psi^{\beta}}\Big].  \label{Cmetric3}
\end{equation}

Let $p$ be a fixed point on a complete K\"ahler manifold $(X, g)$, and $d_{g}$ the distance function from $p$. We define the order of growth of a holomorphic function $f \in \mathcal{O}(X)$ (in the sense of Hadamard) as
\begin{equation}
    d_{H,f}=\inf\{\rho>0 \ |\ \text{there exist}\  C_1, C_2>0\ \text{so that}\ |f(q)| \leq C_1 e^{C_2(d_g(q, p)+1)^{\rho}}\ \text{for any}\ q \in X\}.
\end{equation}
Then we introduce a notion of the minimal $H$-(Hadamard) order on $(X, g)$ as
\begin{equation}
    d_{Hmin}(X, g)=\inf \{\,d_{H,f}\ |\ \text{any nonconstant } f \in \mathcal{O}(X)\,\}.
\end{equation}

\begin{proposition}\label{holoCala}
Let $E$ be the indecomposable rank $2$ vector bundle of degree zero defined in (\ref{Edef}) endowed with the Hermitian metric defined in (\ref{Hdef}). Let $X$ be the total space of $E$.  
Then for any complete K\"ahler metric $g$ on $X$ of the Calabi type, $d_{Hmin}(X, g) \geq 2$. Moreover, there exists a complete K\"ahler metric of the Calabi type such that the corresponding $d_{Hmin} =2$.
\end{proposition}

\begin{proof}[Proof of Proposition \ref{holoCala}]
In our case $t=|v^1-yv^2|^2+k|v^2|^2$. We observe that $|v^1-yv^2|^2-k|v^2|^2=-t$ when $v^1=yv^2$ and $|v^1-yv^2|^2-k|v^2|^2=t$ when $v^2=0$. Since $\omega_g$ is positive-definite, we conclude that $tu'(t) <4k$. Moreover, $u'(t)>0$ and $tu''(t)+u'(t)>0$. The integral in (\ref{fiberdist}) can be bounded from above. Indeed, for $x>1$, we have
\begin{equation}
    \int_1^{x} \sqrt{tu''(t)+u'(t)} \frac{1}{\sqrt{t}} \,dt \le \left(\int_1^{x} (tu'(t))' \,dt\right)^{\frac{1}{2}} \left(\int_1^{x} \frac{1}{t} \,dt\right)^{\frac{1}{2}} \le 2\sqrt{k \ln x}.  \label{roughest}
\end{equation}

Set
\[
r_u(t)=\frac{1}{2}\int_0^t
\frac{\sqrt{u'(s)+su''(s)}}{\sqrt s}\,d s.
\]
Then $r_u(t(q))=d_g(q,\Sigma)$ is the distance to the zero section. Along any curve, we obtain
\[
d s_g^2\ge
\frac{u'(t)+tu''(t)}{4t}(d t)^2.
\]
Thus $|d r_u|\le d s_g$ away from the zero section. A radial curve in a fixed fiber
attains this lower bound. For any fixed $p\in\Sigma$, it follows that
\[
r_u(t(q))\le d_g(q,p)\le r_u(t(q))+\operatorname{diam}(\Sigma,\omega_0).
\]

Now we consider the growth of any nonconstant $f \in \mathcal{O}(X)$ with respect to the distance $d$ induced by a Calabi-type metric $g$. According to Proposition \ref{holoDPS}, it suffices to estimate the growth of $f(z_2) \in \mathcal{O}(\mathbb{C})$ along each fiber. Consider the Taylor expansion $f(\xi)=\sum_{j=0}^{\infty}c_j\xi^j$. Choose $m \ge 1$ with $c_m\ne 0$.
By Cauchy's estimate,
\[
\sup_{|\xi|=R}|f(\xi)|\ge |c_m|R^m.
\]
Note that points represented by $(z,v^1,v^2)=(0,0,\xi)$ with $|\xi|=R$
have $t=kR^2$ and lie in a geodesic ball centered at $p$ of radius
$C_0+C_1\sqrt{\ln R}$ for $R$ sufficiently large by \eqref{roughest}. If $f$ satisfies $d_{H, f}<\rho$ for any given $\rho<2$, then 
\[
m\ln R+\ln|c_m|
\le C_2+C_3(\ln R)^{\rho/2},
\]
which is impossible as $R\to\infty$. Hence we conclude that $d_{Hmin} \geq 2$.

It remains to construct a Calabi-type metric such that $d_{Hmin} =2$. We consider a smooth positive function $h(t)$ on $[0, +\infty)$ so that $h(t) \equiv 1$ near $0$ and 
\begin{equation*}
h(t) =\frac{1}{t\ln t (\ln \ln t)^2}\ \text{for}\ t \in [\alpha+1, +\infty). 
\end{equation*}
Here $\alpha$ is a sufficiently large constant such that $\ln \ln t>2$ when $t >\alpha$. Then $C=\int_0^{+\infty} h(t)\,dt$ is finite. Let 
\begin{equation}
(t u'(t))' =\frac{2k}{C} h(t) \qquad \text{for}\ t \geq 0.    \label{uchoose3}
\end{equation}
Then the coefficient of the $\sqrt{-1}dz \wedge d\overline{z}$ term in (\ref{Cmetric3}) is positive. Note that $t \le |a|^2+|b|^2 \le kt$ for $k\ge 1$ and $tu''(t) +u'(t) > 0$. Hence $u''(t)A +u'(t) B$ is positive definite. We further estimate the geodesic distance by (\ref{fiberdist}). For sufficiently large $x>0$,
\begin{equation*}
    \int_0^x \sqrt{tu''(t)+u'(t)} \frac{1}{\sqrt{t}} \,dt \sim  \int_{\alpha+1}^{x}  \frac{1}{t \sqrt{\ln t} \ln \ln t} \,dt \to \infty.
\end{equation*}
Hence (\ref{Cmetric3}) defines a complete Calabi-type metric.  Finally, $d_{Hmin} =2$ follows from an elementary identity
\begin{equation*}
    \lim_{x \to +\infty} \dfrac{\ln \ln x}{\ln\int_{\alpha+1}^{x}  \frac{1}{t \sqrt{\ln t} \ln \ln t} \,dt}=2.
\end{equation*}
To sum up, (\ref{uchoose3}) is an alternative choice of $u$ in the proof of Theorem \ref{vecmetric} under which nonconstant holomorphic functions grow more `slowly' than in (\ref{uchoice}).
\end{proof}

\section{On vector bundles whose total spaces admit complete K\"ahler metrics with
\texorpdfstring{$BI \geq 0$}{BI >= 0}}\label{sec3}

In this section, we study complete K\"ahler metrics with $BI \geq 0$ on the total space of a line bundle over a complex torus, proving Theorem \ref{res_3_intro}. We begin with a statement of Appell-Humbert's theorem which characterizes any line bundle over a compact complex torus.

Let $V$ be a complex vector space of dimension $n$ and $V=\mathbb{C}\text{-span}\{e_1, \cdots, e_n\}$. Let $H$ be a Hermitian bilinear form on $V \times V$. For any $z \in V$ with $z=z^i e_i$, we define 
\begin{equation}
    h_{i\bar{j}}=H(e_i, e_j)=\operatorname{Re}H(e_i, e_j)+\sqrt{-1}\operatorname{Im}H(e_i, e_j).  \label{Hform1}
\end{equation} Then $\operatorname{Re}H$ is symmetric while $\operatorname{Im}H$ is skew-symmetric, and both are $i$-invariant in the sense of $(\ ,\ )=(i \cdot\ , i \cdot\ )$. We may also consider the corresponding real $(1, 1)$ form
\begin{equation}
    H_{form}=\sqrt{-1}h_{i\bar{j}} dz^i \wedge d\overline{z^j}.  \label{Hform2}
\end{equation} 

Next we pick $\Gamma \subset V$ which is a lattice of real rank $2n$.  Let $V_{\mathbb{R}}$ denote the underlying real vector space of $V$. Then
we have $\Gamma \otimes_{\mathbb{Z}} \mathbb{R}=V_{\mathbb{R}}$. It is possible to define an endomorphism $J$ with $J^2=-id$ on $V_{\mathbb{R}}$ (and its complexification $V_{\mathbb{R}} \otimes_{\mathbb{R}} \mathbb{C}$) so that $V$ is exactly the $\sqrt{-1}$ eigenspace. On the compact complex torus $M=V / \Gamma$, the real $(1, 1)$ form $\frac{1}{2}H_{form}$ defined in (\ref{Hform2}) lies in $H^2(M, \mathbb{Z})$ if and only if $\operatorname{Im}H$ is $\mathbb{Z}$-valued on $\Gamma \times \Gamma$. We refer the reader to \cite[Lemma 3.7]{Schnell_25} for a nice geometric explanation of this fact.

\begin{comment}
We may assume that $\Gamma=\mathbb{Z}\text{-span}\{\widetilde{e}_1, \cdots, \widetilde{e}_n, \cdots, \widetilde{e}_{2n}\}$.
we may relate complex and real coordinates by $z^ie_i = x^{i}\widetilde{e}_i + x^{n+i} \widetilde{e}_{n+i}$, and write (\ref{Hform2}) in terms of $dx^{i}$ and $dx^{n+i}$ with $1 \leq i \leq n$.     
\end{comment}

Let $H$ be a Hermitian bilinear form on $V$ so that $\operatorname{Im}H$ is $\mathbb{Z}$-valued when restricted to $\Gamma \times \Gamma$. Let $\alpha: \Gamma \rightarrow S^1$ so that
\begin{align}
\alpha(\gamma_1+\gamma_2)=e^{i \pi \operatorname{Im} H(\gamma_1, \gamma_2)}\alpha(\gamma_1)\alpha(\gamma_2).   \label{alphasemi}    
\end{align}
Then we may define a line bundle $L(H, \alpha)$ over $M$ as a quotient of $V \times \mathbb{C}$ under the action
\begin{align}
(z, \xi) \rightarrow (z+\gamma, \alpha(\gamma) e^{\pi H(z, \gamma)+\frac{\pi}{2}H(\gamma, \gamma)} \xi).    \label{AHaction}
\end{align} Obviously, (\ref{AHaction}) reduces to (\ref{AHaction_intro}) when $M$ is an elliptic curve. Appell-Humbert's theorem (see \cite[Theorem, p.~19]{Mumford} and \cite[Theorem~5.17, p.~50]{Debarre}) states that any line bundle over $M$ is isomorphic to $L(H, \alpha)$ for some $H$ and $\alpha$ in \eqref{Hform2} and \eqref{alphasemi}. Note that $h(z, \xi)=e^{-\pi H(z, z)}|\xi|^2$ is invariant under the action (\ref{AHaction}) and descends to a Hermitian metric on $L(H, \alpha)$. Moreover, its curvature form $\Theta(L)=-\sqrt{-1}\pi H_{form}$ and a representative of $c_1(L)$ is $-\operatorname{Im}H$.

\begin{comment}
\begin{question}
    Let $X$ be the total space of a line bundle $L$ over an a complex torus $M$ of dimension $n-1$. Assume that $X$ admits a complete K\"ahler metric $g$ with $BI \geq 0$. then the degree of $L$ must be zero. Moreover, the universal covering $(\widetilde{X}, \widetilde{g})$ is biholomorphically isometric to $(\mathbb{C}^{n-1} \times \mathbb{C}, g_{e} + g_1)$. Here $g_e$ is the standard Euclidean metric on $\mathbb{C}^{n-1}$ and $g_1$ is a complete K\"ahler metric on $\mathbb{C}$ with nonnegative Gauss curvature. In particular, we have
    \begin{enumerate}[label=(\arabic*)]
        \item  If $\alpha \equiv 1$ (i.e. $L$ is globally trivial), $g_1$ can be any complete K\"ahler metric on $\mathbb{C}$ with nonnegative Gauss curvature.
        \item  If there exists some $\gamma_k \in \pi_1(T)$ so that $\alpha(\gamma_k)=e^{2\pi \theta_k i}$ for some irrational number $\theta_k \in [0, 1)$, then 
        $g_1$ is rotationally symmetric.
        \item  If there exists a set of generators $\{\gamma_1, \cdots, \gamma_{2n-2} \}$ of $ \pi_1(T)$ so that $\alpha(\gamma_k)=e^{2\pi \theta_k i}$ for each $\theta_k \in [0, 1)$ rational number. Let $m$ denote the least common multiple of the denominators of $\theta_k\, (1 \leq k \leq 2n-2)$ in the reduced form. Then $g_1$ is $\mathbb{Z}_m$-invariant.
    \end{enumerate}
\end{question}    
\end{comment}

We have a simple observation on a class of superharmonic functions and recall a standard fact in linear algebra.

\begin{lemma}\label{nonzero}
    Fix an integer $n \geq 1$. Let $f(z)$ be an entire holomorphic function on $\mathbb{C}^n$ which is not identically zero, and $u$ be a positive smooth function on $\mathbb{C}^n$. We assume that $\sqrt{-1} \partial \overline{\partial} \log u \leq 0$ everywhere on $\mathbb{C}^n$, and $u(z)|f(z)|^2$ is bounded between two positive constants on $\mathbb{C}^n$. Then $\log u$ is harmonic and $u(z)|f(z)|^2$ is constant.
\end{lemma}

\begin{comment}
    \begin{lemma}\label{nonzero}
    Given a fixed integer $n \geq 1$. Let $f(z)$ be an entire holomorphic function on $\mathbb{C}^n$ which is not identically zero, and $u$ be a positive smooth function on $\mathbb{C}^n$. We assume either of the following
     \begin{enumerate}[label=(\arabic*)]
      \item   $\sqrt{-1} \partial \overline{\partial} \log u \leq 0$ everywhere on $\mathbb{C}^n$, and $u(z)|f(z)|^2$ is bounded between two positive constants on $\mathbb{C}^n$.
      \item   $\Delta \log u \leq 0$ everywhere on $\mathbb{C}^n$, and $u(z)|f(z)|^2$ is doubly periodic (invariant under $z \rightarrow z+\gamma$ for any $\gamma \in \Gamma$ for a lattice $\Gamma$ of real rank $2n$ in $\mathbb{C}^n$).
     \end{enumerate}  
     Then $\log u$ is harmonic and $u(z)|f(z)|^2$ is constant.
\end{lemma}
\end{comment}

\begin{proof}[Proof of Lemma \ref{nonzero}]
  Note that $f$ is nowhere vanishing. Then $-\log (u(z)|f(z)|^2)$ is a bounded plurisubharmonic function. Therefore, $u(z)|f(z)|^2$ is constant. Recall that any plurisubharmonic function with sub-logarithmic growth on $\mathbb{C}^n$ is constant. Therefore $\Delta \log u=-\Delta \log |f|^2=0$.
\end{proof}

\begin{comment}
       (2) Let $Z(f)$ denote the zero set of an entire function $f$. Note that $Z(f)$ is a union of analytic hypersurfaces (with multiplicity) in $\mathbb{C}^n$. By the Poincar\'e--Lelong formula (\cite[p.~388]{GH}), we have
   \[
   \sqrt{-1}\partial  \overline{\partial} \log (u |f|^2)=\sqrt{-1} \partial  \overline{\partial}  \log u, \ \ \ z \in \mathbb{C}^n \setminus Z(f).
   \]
   By our assumption, $\sqrt{-1} \partial  \bar{\partial}  \log u$ is doubly-periodic in $\mathbb{C}^n \setminus Z(f)$. As it is smooth, it is in fact doubly-periodic on $\mathbb{C}^n$. As
   \(
   \sqrt{-1} \partial  \bar{\partial}  \log u \wedge \omega_0^{n-1}=\frac{1}{4n}(\Delta \log u)\, \omega_0^{n}
   \) where $\omega_0$ the standard K\"ahler form on $\mathbb{C}^n$, $\Delta \log u$ is also doubly periodic. But then $\Delta \log u=0$ everywhere as the integral of $\Delta \log u$ is zero in a fundamental parallelogram (corresponding to $\Gamma$). Now that $\log u$ is harmonic, we get that $\log(u(z)|f(z)|^2)$ is subharmonic. As it is doubly-periodic and attains the maximum at some point, $u(z)|f(z)|^2$ is constant. 
\end{comment}

\begin{lemma}\label{detlemma}
    Let $a_1,a_2,\cdots,a_n$ be $n$-dimensional column vectors, and $A=a_1 \overline{a_1}^T+ \cdots +a_n \overline{a_n}^T$. Then $A=B \overline{B}^T$ where $B=(a_1,a_2,\cdots,a_n)$. In particular, $\det A=|\det B|^2$.
\end{lemma}

\begin{proof}[Proof of part (1) of Theorem \ref{res_3_intro}]

We divide the proof into four steps. Let $E=\oplus_{j=1}^s E_j$ where each $E_j$ is an indecomposable vector bundle over the elliptic curve $\Sigma$. Assume that $X$, the total space of $E$, admits a complete K\"ahler metric $\omega$ with $BI \geq 0$.

\vskip 0.2cm
\textbf{Step 1:} Each $E_j$ satisfies $\deg E_j \geq 0$.
\vskip 0.2cm

Take $T^{1,0} X \big |_{\Sigma}$ and $T^{1,0} \Sigma$, we note that the normal bundle of $\Sigma$ in $X$ can be identified with $E$. Then we get the following exact sequence
\[
0 \to T^{1,0} \Sigma \to T^{1,0} X \big |_{\Sigma} \to E \to 0.
\]
For $U\in T^{1,0} \Sigma$ and $V \in T^{1,0} X \big |_{\Sigma}$, we observe
\(
\langle \Theta_{T^{1,0} X |_{\Sigma}} (U,\overline U)V,V \rangle=R(U,\overline U,V,\overline V)\geq 0.
\) Here we use $BI(\omega) \geq 0$. Then $T^{1,0} X \big |_{\Sigma}$ is Griffiths semipositive, so is $E$. As each summand $E_j$ can be seen as a quotient bundle of $E$, it follows that $E_j$ is Griffiths semipositive. Thus
\(
 \deg E_j=\frac{\sqrt{-1}}{2\pi}\int_{\Sigma} \operatorname{tr} \Theta_{E_j} \geq 0.
\)

\vskip 0.2cm
\textbf{Step 2:} We show at least one of $E_j$ with $1 \leq j \leq s$ satisfies $\deg E_j = 0$.
\vskip 0.2cm

If every $\deg E_j>0$, then $E$ is ample by \cite[Theorem 1.3]{Harts1971}. Moreover, by \cite[Theorem 2.6]{Umemura}, $E$ admits a Hermitian metric $h$ such that $E$ is Griffiths positive. We consider $\rho(z, w)=|w|^2_{h}$ which is smooth on $X$ and denote $D_R=\{\rho <R^2\} \subset X$. For any $p \in \partial D_R$, we choose a local coordinate $\{z\}$ such that $z(p_0)=0$ and a holomorphic normal frame $\{e_{\alpha}\}_{\alpha=1}^r$ of $E$ such that $h_{\alpha \overline{\beta}}(0)=\delta_{\alpha \beta}$ and $dh_{\alpha \overline{\beta}}\,(0)=0$. Then
\[
\Theta_{(E, h)} (p_0)=-\frac{\partial^2 h_{\alpha \overline{\beta}}}{\partial z \partial \overline{z} }\big|_{p_0} dz \wedge d\overline{z} \otimes  e^{\alpha} \otimes  (e^{\beta})^{\ast}.
\]
For $q=(z, \xi^{\alpha} e_{\alpha})$ near $p$, we may write $\rho=h_{\alpha \overline{\beta}} \xi^{\alpha}\overline{\xi^{\beta}}$. We set a horizontal vector field $U=\left.\frac{\partial}{\partial z} \right|_{p}$. At $p$, we obtain
\[
d\rho(U)=0, \qquad \sqrt{-1} \partial \overline \partial \rho(U, \overline{U})=(h_{\alpha \overline{\beta}})_{z \overline{z}} \xi^{\alpha} \overline{\xi^{\beta}}=-R_{E}(U, \overline{U}, \xi^{\alpha} e_{\alpha}, \overline{\xi^{\alpha} e_{\alpha}})<0.
\]
Consider the Taylor expansion of $\rho$ with respect to $z$ near $p$. There exist constants $a \in \mathbb{C}$ and $c>0$, such that
\begin{equation}
\rho(z,w_0)=R^2 +\operatorname{Re}(az^2)-c|z|^2 +O(|z|^3).   \label{rho_taylor_1}    
\end{equation}
Let $\xi(z)=(1-\frac{a}{2R^2}z^2) \xi(p)$. We consider $f(z)=(z, \xi(z))$. It follows from  \eqref{rho_taylor_1} that
\begin{equation}
\rho(f(z))=R^2-c|z|^2 +O(|z|^3).   \label{rho_taylor_2}
\end{equation}
Choose $\varepsilon>0$ sufficiently small such that $\rho(f(z))<R^2$ for any $0<|z|< \varepsilon$. We conclude that for any $R>0$ and any boundary point $p$ of $D_R$, there exists a holomorphic disc $\widehat{f}: \Delta_{\varepsilon} \rightarrow \overline{D_{R}}$ with $\widehat{f}(0)=p$ and $\widehat{f}(\Delta_{\varepsilon}\backslash \{0\}) \subset D_{R}$. 

It is well-known from Wu \cite{Wu1979} or Ni--Tam \cite{NT2003} that $X$ admits a nontrivial (smooth) PSH function $\mathcal{B}$ with linear growth. Then there exists $R>0$ such that $\mathcal{B}$ is non-constant on $D_R$. Thus there is a point $p \in \partial D_R$ such that $\mathcal{B}(p)=\max_{q \in \overline{D_{R}}} \mathcal{B}(q)$ by the strong maximum principle. We choose a holomorphic disc $\widehat{f}$ in the above. Note that $\mathcal{B} \circ \widehat{f}$ is subharmonic and attains its maximum at $p$. By the strong maximum principle $\mathcal{B} \circ \widehat{f}$ is constant. But then there exists $q \in D_R$ with $\mathcal{B}(q)=\max_{ \overline{D_{R}}} \mathcal{B}$, which forces $\mathcal{B}$ to be constant.

\vskip 0.2cm
\textbf{Step 3:} Each $E_j$ satisfies $\deg E_j=0$.
\vskip 0.2cm

Consider $E=E_0 \oplus E_+$, where $E_0$ is a direct sum of indecomposable bundles with zero degree and $E_+$ is a direct sum of indecomposable bundles with positive degree. Let $Y$ be the total space of $E_{+}$. Then the normal bundle $N \coloneqq N_{Y/X} \simeq \pi_Y^*(E_0)$ where $\pi_Y: Y \rightarrow \Sigma$. Hence $N$ is Griffiths non-negative, and it admits a (quotient bundle) metric $h_N$ satisfying $\Theta_{h_N}\geq 0$. Note that $\det N \simeq \pi_Y^* (\det E_0)$. $\deg (\det E_0)=0$.
Therefore, $\det N$ admits a unitary flat Hermitian metric $h_1$ and 
\(
\varphi=\log \frac{h_1}{\det h_N}
\)
defines a smooth function on $Y$. Since 
\(
\sqrt{-1} \partial \overline \partial \varphi =\sqrt{-1}\operatorname{tr}\Theta_{h_N} \ge 0,
\)
we conclude that $\varphi$ is PSH on $Y$, hence a constant by the argument in Step 2. As $\det N$ is unitary flat and $N$ is Griffiths semipositive, we obtain that $N$ is also unitary flat. It follows from the comparison between curvatures of $T^{1, 0}X \big |_Y$ and $N$ \cite[p.~78]{GH} that the second fundamental form of $Y$ vanishes. Then $Y$ is totally geodesic in $X$, and then admits a complete K\"ahler metric with $BI \ge 0$ induced from $X$. However, after we apply the argument in Step 2 on $E_{+}$, we arrive at a contradiction that at least one summand of $E_{+}$ has degree zero. To sum up, every summand in the decomposition $E=\oplus E_j$ has degree zero.

\vskip 0.2cm
\textbf{Step 4:} $E=\oplus_{j=1}^r L_j$ with each $L_j$ is a line bundle of $\deg L_j=0$.
\vskip 0.2cm

By Step 1, $E$ is Griffiths semipositive. In the meantime, $\deg E=0$ by Step 3. It follows that $E$ admits a Hermitian metric with vanishing Chern curvature, i.e., is unitary flat. It is a standard fact that $E$ is associated with a unitary representation $\rho: \pi_1(\Sigma) \rightarrow U(r)$. As $\pi_1(\Sigma)$ is abelian, $\rho$ is diagonalizable. Therefore $E$ is a direct sum of line bundles, each of which has degree zero by Step 3.
\end{proof}

\begin{lemma}\label{nonflat}
    Let $(z, \xi_2, \ldots, \xi_{n})$ denote global holomorphic coordinates on $\mathbb{C}^n$. Assume that
    \[
    \omega=\sqrt{-1}\left( u(z)dz \wedge d\bar z + \sum_{j=2}^n d\xi_j \wedge d\overline{\xi}_j \right)
    \] is a complete non-flat K\"ahler metric on $\mathbb{C}^n$ with $BI \geq 0$. That is to say: $\Delta_{z} \log u(z) \leq 0$ for any $z \in \mathbb{C}$ but not identically zero. Then any holomorphic isometric map $F \in \operatorname{Aut}(\mathbb{C}^n)$ can be written as
    \(
    F(z,\xi)=(az+b, F_2(\xi)),
    \)
    where $a \in \mathbb{C}^{\ast}$ and $b \in \mathbb{C}$ such that $|a|^2 u(az+b)=u(z)$ and $F_2$ is a holomorphic isometry on $\mathbb{C}^{n-1}$ with respect to the Euclidean metric.
\end{lemma}

\begin{proof}[Proof of Lemma \ref{nonflat}]
    Set $F(z,\xi)=(F_1(z,\xi),F_2(z,\xi))$. As $F$ preserves the Ricci form of $\omega$, we derive that
\begin{align*}
    \partial_z \partial_{\overline{z}} \log u(F_1(z,\xi)) \left| \frac{\partial F_1}{\partial z}\right|^2= \partial_z \partial_{\overline{z}} \log u(z),\ \ \partial_z \partial_{\overline{z}} \log u(F_1(z,\xi)) \left| \frac{\partial F_1}{\partial \xi_j}\right|^2=0, \ \ j=2,3,\ldots,n.
\end{align*}
As $\partial_z \partial_{\overline{z}} \log u$ is negative in an open set, $\frac{\partial F_1}{\partial \xi_j}=0$ correspondingly. Here we used the fact that $F_1:\mathbb{C}^n \to \mathbb{C}$ is surjective. Hence $\frac{\partial F_1}{\partial \xi_j}=0$ for any $j=2,\ldots,n$ and $F_1(z,\xi)=F_1(z)$. Similarly, we have $F^{-1}(z,\xi)=(G_1(z),G_2(z,\xi))$ where $G_1(F_1(z))=z$. It follows that $F_1(z)=az+b$. As $F^{\ast}\omega=\omega$, we can show that $F_2(z,\xi)$ is independent of $z$ after comparing the coefficients of $dz \wedge d\overline{\xi_i}$. Therefore, $F_2(\xi)$ is a holomorphically isometric map of $\mathbb{C}^{n-1}$. 
\end{proof}

We restate part (2) of Theorem \ref{res_3_intro} separately in the following.

\begin{lemma}\label{BIintro_part2}
     Let $L(H,\alpha)$ be a line bundle over a compact complex torus $\mathbb{C}^{n-1} / \Gamma$ 
     as in \eqref{AHaction}. Assume that $X$, the total space of $L(H,\alpha)$, admits a complete K\"ahler metric $\omega$ such that the lifted metric on its universal cover $\mathbb{C}^n$ has the following form
     \begin{equation}
         \widetilde{\omega}=\sqrt{-1} (u \, df_1 \wedge d\overline{f_1}+\sum_{j=2}^n df_j \wedge d\overline{f_j}).  \label{lift_split_more}
     \end{equation}
     Here $F(z,\xi)=(f_1(z,\xi),f_2(z,\xi),\cdots,f_n(z,\xi)) \in \operatorname{Aut}(\mathbb{C}^n)$, $u(w)|dw|^2$ is a complete K\"ahler metric with nonnegative Gauss curvature on $\mathbb{C}$, and $u(z, \xi) :=u \circ f_1(z,\xi)$. Then $H=0$.
\end{lemma}

\begin{proof}[Proof of Lemma \ref{BIintro_part2}]
    We focus on the zero section $\{\xi=0\}$ in $X$ (or its lifting in $\mathbb{C}^n$). Let
    \begin{equation*}
        df_j= f_{j,1}(z)dz_1 +\cdots +f_{j,n-1}(z) dz_{n-1}+ f_{j,n}(z) d\xi
    \end{equation*}
    for $j=1,2,\cdots,n$. By Lemma \ref{detlemma}, we have
    $\det(g)=u(z)|h(z)|^2 >0$ where $u(z)=u(f_{1}(z, 0))$ and $h(z)$ denotes the restriction of the complex Jacobian determinant of $F$ along $\{\xi=0\}$. As $\widetilde{\omega}$ is invariant with respect to \eqref{AHaction}, we observe that the volume form $u(z)|h(z)|^2 e^{\pi H(z,z)}$ is invariant under $z \to z+\gamma$ for any $\gamma \in \Gamma$. 
    Note that 
    \begin{equation}
    \Delta \log u(z)= (\Delta \log u) \sum_{i=1}^{n-1}\bigl|\frac{\partial f_1}{\partial z_i}\bigr|^2 \leq 0  \label{PSH_1}   
    \end{equation} as $u(w)|dw|^2$ has nonnegative Gauss curvature. In fact, we get that $-\log u$ is PSH on $\mathbb{C}^{n-1}$.

    For any fixed $z_0=(z_{0, 1}, \cdots, z_{0, {n-1}})$, 
    \begin{equation}
        \Delta_t \ln \Big(u(t z_0)  e^{\pi H(z_0,z_0)|t|^2}\Big)=(\Delta \ln u) \left|\sum_{k=1}^{n-1} \frac{\partial f_1}{\partial z_k} z_{0, k}\right|^2+4\pi H(z_0,z_0).  \label{restrict_1}
    \end{equation}

    Assume that $H(z_0,z_0) < 0$. Recall that $u(t z_0)|h(t z_0)|^2 e^{\pi H(z_0,z_0)|t|^2}$ is bounded between two positive constants. By the $n=1$ case of Lemma \ref{nonzero}, it follows from \eqref{restrict_1} that $H(z_0,z_0)\geq 0$. This contradiction shows that $H \geq 0$.

    It remains to show $H=0$. If $\widetilde{\omega}$ is flat, then $\Delta \log u \equiv0$. It follows that 
    \[
    \Delta \log \left( u(z) |h(z)|^2 e^{\pi H(z,z)} \right)=4\pi \operatorname{tr}H \ge 0.
    \]
    Recall that $u(z) |h(z)|^2 e^{\pi H(z,z)}$ is invariant under $z \to z+\gamma$. Then $\operatorname{tr} H=0$. Hence $H=0$. We only need to study the non-flat case. It follows from Lemma \ref{nonflat} that for any $\gamma \in \Gamma$, there exist $a_\gamma \in \mathbb{C}^*$, $b_{\gamma} \in \mathbb{C}$ such that 
   \begin{equation}\label{f1}
    f_1(z+\gamma,0)= a_{\gamma}f_1(z,0) +b_{\gamma}.
    \end{equation}
    Write $f_1(z)=f_1(z,0)$ for simplicity. We argue by two cases.

    \vskip 0.2cm
    \textbf{Case 1:} If $a_{\gamma}=1$ for all $\gamma \in \Gamma$, then it follows from \eqref{f1} that $f_1(z)$ has at most linear growth. In this case  $f_1(z)=az+b$ where $a \in \mathbb{C}^{n-1}$ and $b\in \mathbb{C}$. 
    \vskip 0.2cm
    \textbf{Case 2:} If $a_{\gamma} \ne 1$ for some $\gamma \in \Gamma$, then there exists $z_0$ such that $z_0=a_{\gamma}z_0 +b_{\gamma}$. Set $T_{\delta}(w)=a_\delta w+b_\delta$. Note that the maps $T_{\delta}$ commute with one another. Hence $T_{\delta}(z_0)$ is also a fixed point of $T_{\gamma}$. It follows that $T_\delta(z_0)=z_0$ since the solution of $z=a_{\gamma}z +b_{\gamma}$ is unique. Again by Lemma \ref{nonflat}, we have $u(z_0)|a_{\delta}|^2=u(z_0)$ and then $|a_{\delta}|=1$. It follows from \eqref{f1} that for any $\delta \in \Gamma$,
    \[
    f_1(z+\delta)-z_0 =a_{\delta} f_1(z)+b_{\delta} -z_0 = a_{\delta}(f_1(z)-z_0).
    \] Thus $|f_1(z)-z_0|$ is bounded and $f_1(z)$ must be constant.

    If $f_1(z)$ is constant, $u(z)=u \circ f_1(z)$ is also constant. In Case 1, if $f_1(z)=az+b$ is nonconstant, then there exists some $1 \leq p \leq n-1$ such that $f_{1,p}(z)$ is a non-zero constant. Checking the coefficients in front of $dz_p \wedge d\overline{z}_p$ in \eqref{lift_split_more}, we get that
   \[
   u(z)|f_{1, p}|^2+|f_{2, p}(z)|^2+\cdots+|f_{n, p}(z)|^2
   \]
    is $\Gamma$-periodic. As each $f_{k, p}$ with $2 \leq k \leq n$ has to be constant, $u(z)|f_{1,p}|^2$ is $\Gamma$-periodic. Then $u(z)$ must be constant by the maximum principle. Therefore, in either case, we obtain that $u(z)$ is constant.
    
    Now we consider the coefficients in front of $d\xi \wedge d\overline{\xi}$ in \eqref{lift_split_more}. By the invariance with respect to \eqref{AHaction}, we have
    \[
    \bigl(u|f_{1, n}(z)|^2+|f_{2, n}(z)|^2+\ldots+|f_{n, n}(z)|^2\bigr)e^{\pi H(z, z)}
    \]
    is $\Gamma$-periodic. Suppose that there exists $z_0$ with $H(z_0,z_0) > 0$. Consider its restriction along $\{ tz_0 | t \in \mathbb{C}\}$. Then 
    \begin{equation}
        (u|f_{1,n}(tz_0)|^2+ |f_{2,n}(tz_0)|^2 +\cdots +|f_{n,n}(tz_0)|^2) e^{\pi H(z_0,z_0)|t|^2}  \label{clinebdd}
    \end{equation}
    is bounded. Taking $|t| \to +\infty$, one has 
    \[
    u|f_{1,n}(tz_0)|^2+ |f_{2,n}(tz_0)|^2 +\cdots +|f_{n,n}(tz_0)|^2 \to 0.
    \]
    Hence $f_{1,n}(tz_0)=f_{2,n}(tz_0)=\cdots=f_{n,n}(tz_0)=0$ via the maximum principle, which contradicts the fact that $F \in \operatorname{Aut}(\mathbb{C}^n)$.  
    To sum up, we have proved $H \equiv 0$.
\end{proof}

\begin{proof}  [Proof of part (3) of Theorem \ref{res_3_intro}]

Assume that $X$, the total space of $L=L(0, \alpha)$, admits a complete K\"ahler metric $\omega$ with $BI \geq 0$.  We study the lifted K\"ahler metric $\widetilde{\omega}$ on the universal cover $\mathbb{C}^{n}$, which is of the form \eqref{lift_split_more}. Recall that
\begin{align}
    df_k=f_{k,1}(z,\xi) dz_1 +f_{k,2}(z,\xi)dz_2+\cdots +f_{k,n-1}(z,\xi)dz_{n-1}+f_{k,n}(z,\xi) d\xi, \ \ \ \ k=1,\cdots, n.
    \label{df1fd2gen}
\end{align}
The local components of $\widetilde{\omega}$ read
\begin{equation*}
    g_{i \overline j} = u(z,\xi)f_{1,i}(z,\xi)\overline{f_{1,j}(z,\xi)} +\sum_{k=2}^n f_{k,i}(z,\xi)\overline{f_{k,j}(z,\xi)}.
\end{equation*}
It follows from Lemma \ref{detlemma} that there exists a holomorphic function $h(z,\xi)$ such that 
\begin{equation}
    \det g\,(z, \xi) =u(z,\xi) |h(z,\xi)|^2 >0.
\end{equation}
According to \eqref{AHaction}, $\det g\,(z, \xi)$ is determined by its restriction on $K \times S_{|\xi|} \subset \mathbb{C}^{n}$. Here $K$ can be chosen as the closed fundamental parallelotope which corresponds to $\Gamma$, and $S_{|\xi|}$ denotes the circle of radius $|\xi|$ at the origin of $\mathbb{C}$. By Lemma \ref{nonzero}, for any fixed $\xi \in \mathbb{C}$, there exists some constant $C(\xi)>0$ such that
\begin{equation}
    u(z,\xi)= \frac{C(\xi)}{|h(z,\xi)|^2}.
\end{equation}

For any fixed $\xi$, we know that $g_{i \overline i}$ is bounded for $i=1,\cdots,n-1$, and $f_{k,i}(z,\xi)$ is bounded. Hence $f_{k,i}(z,\xi)=f_{k,i}(\xi)$ which is independent of $z$ for any $k=2,\cdots, n$. Moreover, for a fixed $\xi$, 
\begin{equation}
    u(z,\xi) |f_{1,i}(z,\xi)|^2= 
    C(\xi) \big|\frac{f_{1,i}(z,\xi)}{h(z,\xi)}\big|^2
\end{equation}
is bounded. Hence there exists $C_i(\xi)$ with $f_{1,i}(z,\xi)=C_i(\xi) h(z,\xi)$ for $1\leq i\leq n-1$. By a similar analysis on $g_{n\overline{n}}$, we have $f_{1,n}(z,\xi)=C_n(\xi) h(z,\xi)$, and $f_{k,n}(z,\xi)=f_{k,n}(\xi)$ for $2 \leq k \leq n$. By (\ref{df1fd2gen}), we have $\frac{\partial f_k}{\partial z_i}=f_{k, i}(\xi)$ and $\frac{\partial f_k}{\partial \xi}=f_{k, n}(\xi)$. Therefore, $f_{k,1}(\xi),\cdots,f_{k,n-1}(\xi)$ must be constant. It follows that for $i=1,\cdots,n-1$,
\begin{equation}
    g_{i \overline i}=u(z,\xi)|f_{1,i}(z,\xi)|^2 +\sum_{k=2}^n |f_{k,i}(z,\xi)|^2 =C(\xi)|C_i (\xi)|^2+\sum_{k=2}^n |f_{k,i}|^2.  \label{gexpress1}
\end{equation}
For any $1 \leq i \leq n-1$, $C(\xi)|C_i(\xi)|^2$ must be constant by the K\"ahler condition $\frac{\partial g_{i\overline{i}}}{\partial \xi}=\frac{\partial g_{\xi\overline{i}}}{\partial z_i}$. 

\textbf{Case 1:} $C_1(\xi),\cdots,C_{n-1}(\xi)$ are not all  identically zero. Now each $\widehat{C}_i \coloneqq C(\xi)|C_i(\xi)|^2$ is constant. For simplicity, we assume that $\widehat{C}_1$ is nonzero, hence $C_1(\xi)$ is nowhere vanishing. As $|\frac{C_i(\xi)}{C_1(\xi)}|^2=\frac{\widehat{C_i}}{\widehat{C_1}}$ is constant, we have $\frac{C_i(\xi)}{C_1(\xi)}=\mu_i$ is constant for $1 \leq i \leq n-1$. Let $\frac{C_i(\xi)}{C_1(\xi)}=\mu_i$ and $\lambda(\xi)=\frac{C_n(\xi)}{C_1(\xi)}$. Then 
\begin{equation}
    \omega_g= \sqrt{-1}\Big[ \widehat{C}_1 (\sum_{i=1}^{n-1} \mu_i dz_i+\lambda(\xi)d\xi) \wedge \overline{(\sum_{i=1}^{n-1} \mu_i dz_i+\lambda(\xi)d\xi)} +\sum_{j=2}^n df_j \wedge d\overline{f_j} \Big].  \label{gexpress2}
\end{equation}
(\ref{gexpress2}) corresponds to a complete flat K\"ahler metric on $\mathbb{C}^n$. Moreover, the induced metric along the $\xi$ direction does not need to be flat. In the case of $n=2$, we may write \eqref{gexpress2} in a more explicit way. Precisely, we assume that $C=C(\xi)|C_1(\xi)|^2>0$ and $C_1(\xi)$ is nowhere vanishing.  (\ref{gexpress2}) can be reduced to 
\begin{equation}
    \omega= \sqrt{-1}\Big[ C (dz+\lambda(\xi)d\xi) \wedge (\overline{dz+ \lambda(\xi)d\xi}) +(f_{2,1} dz+f_{2,2}(\xi) d\xi) \wedge (\overline{f_{2,1} dz+f_{2,2}(\xi) d\xi})\Big].  \label{gexpress2_n2}
\end{equation}
We recall an elementary formula as follows.
\begin{equation*}
    (a\eta_1 +b \eta_2) \wedge (\overline a \overline{\eta_1} +\overline b \overline{\eta_2}) + (\overline b\eta_1 -\overline a \eta_2) \wedge (b \overline{\eta_1} - a \overline{\eta_2})
    =(|a|^2+|b|^2)(\eta_1 \wedge \overline{\eta_1}+ \eta_2 \wedge \overline{\eta_2}).
\end{equation*}
Let $\eta_1=\sqrt{C}(dz+\lambda(\xi) d\xi)$ and $\eta_2= f_{2,1} dz+f_{2,2}(\xi) d\xi$ while $a=\sqrt{C}$ and $b=\overline{f_{2,1}}$. Then \eqref{gexpress2_n2} becomes
\begin{equation}
    \omega=\sqrt{-1}\Big[(C+|f_{2,1}|^2) d(z + h_1(\xi)) \wedge d\overline{(z + h_1(\xi))} + |h_2 (\xi)|^2 d\xi \wedge d\overline{\xi}\Big]
    \label{gexpress3_n2}
\end{equation}
where 
\begin{equation}
h_1^{\prime}(\xi)=\frac{C \lambda(\xi)+\overline{f_{2, 1}}f_{2,2}(\xi)}{C+|f_{2,1}|^2},\ \ \ h_2(\xi)=\frac{\sqrt{C}}{\sqrt{C+|f_{2,1}|^2}} (f_{2,1}\lambda(\xi)-f_{2,2}(\xi)).
\end{equation}
Note that both $h_1$ and $h_2$ are holomorphic on $\mathbb{C}$ and $h_2$ is nowhere vanishing.

\textbf{Case 2:} $C_1(\xi)=\cdots=C_{n-1}(\xi) \equiv 0$. In this case, $f_{1,i}(z,\xi) \equiv 0$ for $i=1,\cdots,n-1$ and $f_1$ depends only on $\xi$. Then we can rewrite $u(f_1(\xi)) df_1 \wedge d\overline{f_1}$ as $u(f_1(\xi)) |\frac{\partial f_1}{\partial \xi}|^2d\xi \wedge d\overline{\xi}$. 
In this case,
\begin{equation}
    \omega_g=\sqrt{-1}\left(\widetilde{u}(\xi) d\xi \wedge d\overline{\xi}+\sum_{j=2}^n df_j \wedge d\overline{f_j}\right),
    \label{geexpress4}
\end{equation}
where $f_j(z,\xi)=f_{j,1}z_1+\cdots +f_{j,n-1} z_{n-1} +h_j(\xi)$. Moreover, $h_j(\xi)$ is a holomorphic function satisfying $h_j(\xi)=h_j(\alpha(\gamma) \xi)$ for any $\gamma \in \Gamma$. In the case $n=2$, \eqref{geexpress4} can be simplified.
\begin{equation}
    \omega=\sqrt{-1}(\widetilde{u}(\xi) d\xi \wedge d \bar \xi + |f_{2, 1}|^2 (dz + d h(\xi)) \wedge (d\bar z +d\overline{h(\xi)}) ), \label{gexpress4_n2}
\end{equation} where $\widetilde{u}(\xi)=u(f_1(\xi)) |\frac{\partial f_1}{\partial \xi}|^2$ and $h^{\prime}(\xi)=\frac{f_{2,2}(\xi)}{f_{2,1}}$. Therefore, $\widetilde{u}(\xi)$ satisfies $\widetilde{u}(e^{2\pi i \theta_1} \xi)=\widetilde{u}(e^{2\pi i \theta_2} \xi) =\widetilde{u}(\xi)$. Moreover, $\Delta \ln \widetilde{u} \le 0$ and $h(\xi)$ is a holomorphic function with the invariant property $h(e^{2\pi i \theta_1} \xi)=h(e^{2\pi i \theta_2} \xi) =h(\xi)$. In particular, if one of $\theta_1$ and $\theta_2$ is irrational, $h(\xi)$ is forced to be constant and (\ref{gexpress4_n2}) is reduced to 
\begin{equation}
    \omega=\sqrt{-1}( |f_{2, 1}|^2 dz \wedge d\overline{z} +\widetilde{u}(|\xi|) d\xi \wedge d \overline{\xi}).
\end{equation}
If both $\theta_1$ and $\theta_2$ are rational, there exists a minimal positive integer $k$ so that $k\theta_1$ and $k\theta_2$ are integers. $\widetilde{u}(\xi)$ is $\mathbb{Z}_k$-invariant in the sense that $\widetilde{u}(\xi)=\widetilde{u}(e^{\frac{2\pi i}{k}}\xi)$ for any $\xi \in \mathbb{C}$. Similarly, $h(\xi)$ is an entire function of $\xi^k$. Then \eqref{gexpress4_n2} is reduced to
\begin{equation}\label{gexpress4_n2_more}
    \omega=\sqrt{-1}( |f_{2, 1}|^2 d(z+ h(\xi)) \wedge d(\overline{z}+ \overline{h(\xi)}) +\widetilde{u}(\xi) d\xi \wedge d \overline{\xi}).
\end{equation}

It follows from \eqref{gexpress3_n2} and \eqref{gexpress4_n2} that the zero section of $X$ is totally geodesic. However, the fiber is not totally geodesic in general; see Example \ref{negexam} below. Moreover, we can show that any K\"ahler metric with $BI \geq 0$ in part (3) of Theorem \ref{res_3_intro} has a nice form up to an automorphism of $X$. For example, let $X$ be the total space of $L(0, \alpha)$ where $\alpha_i=e^{2\pi i \theta_i}$ and $\theta_i$ is rational for $i=1, 2$. We consider $F \in \operatorname{Aut}(\mathbb{C}^2)$ defined by 
\begin{equation}\label{auto_form_1}
(\widehat{z}, \widehat{\xi})=F(z, \xi)=(z+h(\xi), \xi), \ \ \ (z, \xi) \in \mathbb{C}^2.     
\end{equation}Then we may check that $F$ descends to an element of $\operatorname{Aut}(X)$. In terms of new coordinates $(\widehat{z}, \widehat{\xi})$, $X$ can be realized as the total space of a line bundle $\widehat{L}$ over an elliptic curve defined by $\{\widehat{\xi}=0\}$. In fact, the base elliptic curve $\widehat{\Sigma}$ is isomorphic to $\Sigma$, and $\widehat{L}$ is bundle isomorphic to the original bundle $L(0, \alpha)$. Moreover, if we write the projection map $\widehat{\pi}: \widehat{L} \rightarrow \widehat{\Sigma}$, then $\omega$ in \eqref{gexpress4_n2_more} becomes a product metric on $\widehat{\pi}^{-1}(U)$ for a small open set $U \subset \widehat{\Sigma}$ in the sense that
\[
\omega=\sqrt{-1}( |f_{2, 1}|^2 d\widehat{z} \wedge d\overline{\widehat{z}} +\widetilde{u}(\widehat{\xi}) d\widehat{\xi} \wedge d \overline{\widehat{\xi}}).
\]

\end{proof}

\begin{example}\label{negexam}
Let $X$ be the total space of a line bundle over an elliptic curve $\Sigma$. It follows from (\ref{gexpress4_n2}) that any nonflat complete K\"ahler metric $g$ on $X$ with $BI \geq 0$ must be of the form
\begin{equation*}
    g=\begin{pmatrix}
     1 & \overline{\widetilde{h}(\xi)} \\
     \widetilde{h}(\xi) & \widetilde{u}(\xi)+|\widetilde{h}(\xi)|^2
    \end{pmatrix}.
\end{equation*} 
Here we choose $f_{2, 1}=1$ for simplicity, and $\widetilde{h}(\xi)=h'(\xi)$ for $h$ defined in \eqref{gexpress4_n2}. Let $\widetilde{u}(\xi)=\frac{1}{1+|\xi|^2}$ and $\widetilde{h}(\xi)=\xi$. Note that we choose $\widetilde{u}(\xi)|d\xi|^2$ as Hamilton's cigar soliton on $\mathbb{C}$ in \cite{Hamilton}. We may
choose $\theta_1=\theta_2=\frac{1}{2}$, i.e., $\alpha(\gamma_i)=-1$ for $i=1, 2$. Then the induced metric on the fiber $(|\xi|^2 +\frac{1}{1+|\xi|^2})|d\xi|^2$ has nonpositive Gauss curvature everywhere. To see it, we consider $f(t)=t+\frac{1}{1+t}$ on $[0, \infty)$. Then 
\begin{equation*}
    t(\ln f)'=\frac{2t^2+t}{t^2+t+1}-\frac{t}{1+t},\ \ (t(\ln f)^{\prime})^{\prime}=\frac{t(4t^2 + 7t + 4)}{(1+t)^2(t^2+t+1)^2}\ge 0.
\end{equation*}
\end{example}

\section{Function theory on total spaces of rank \texorpdfstring{$2$}{TEXT} bundles of degree zero}
\label{sec4}

In this section, we focus on rank $2$ bundles of degree zero over an elliptic curve $\Sigma$ defined by (\ref{ecurve}) and prove several function-theoretic results on their total spaces, including Propositions \ref{Xfunction} and \ref{biholoC}.

First of all, we recall an explicit holomorphic bundle isomorphism between two vector bundles which are associated with representations (\ref{gen-rep}) with the same value of $b_2-b_1\tau$.

\begin{lemma}[{\cite[Theorems 4.5 and 4.6]{Iena}}]\label{isobundle}
    Let $X$ (resp.~$Y$) be the total space of $E$ (resp.~$\widetilde{E}$) over the elliptic curve $\Sigma$ defined by (\ref{ecurve}). Suppose $E$ is associated with the representation $\rho_E(\gamma_k)=e^{2\pi i \theta_k}\begin{pmatrix}
        1&b_k\\0&1
    \end{pmatrix}$ and $\widetilde{E}$ is associated with
    $\rho_{\widetilde{E}}(\gamma_k)=e^{2\pi i \theta_k}\begin{pmatrix}
        1&\widetilde{b}_k\\0&1
    \end{pmatrix}$ with $k\in \{1, 2\}$. If $b_2-b_1 \tau=\widetilde{b}_2-\widetilde{b}_1 \tau$, then  $F(z,z_1,z_2)=(z,z_1-b_1z z_2+ \widetilde{b}_1 z z_2,z_2)$ defines a bundle isomorphism between $E$ and $\widetilde{E}$. In particular, if $E$ is indecomposable, then $b_2-b_1 \tau \ne0$.
\end{lemma}
\begin{proof}[Proof of Lemma \ref{isobundle}]
    It is a direct computation to show that:
    \begin{align*}
        &F(z+1, e^{2\pi i \theta_1}
        \begin{pmatrix}
            z_1+b_1 z_2 \\ z_2
        \end{pmatrix})
        =(z+1,e^{2\pi i \theta_1}\begin{pmatrix}
            1&\widetilde{b}_1 \\ 0&1
        \end{pmatrix}
        \begin{pmatrix}
            z_1-b_1z z_2+\widetilde{b}_1 z z_2\\ z_2
        \end{pmatrix}).\\
        &F(z+\tau,e^{2\pi i \theta_2}
        \begin{pmatrix}
            z_1+b_2 z_2 \\ z_2
        \end{pmatrix})=(z+\tau,e^{2\pi i \theta_2}\begin{pmatrix}
            1&\widetilde{b}_2 \\ 0&1
        \end{pmatrix}
        \begin{pmatrix}
            z_1-b_1z z_2+\widetilde{b}_1 z z_2\\ z_2
        \end{pmatrix}).
    \end{align*}
    Therefore $F$ descends to a bundle isomorphism between $E$ and $\widetilde{E}$.
\end{proof}

The following result gives a complete characterization of $\mathcal{O}(X)$ when $X$ is the total space of any rank-$2$ vector bundle of degree $0$.

\begin{proposition}\label{Xfunction}

Let $E$ be a rank $2$ vector bundle of degree $0$ over the elliptic curve $\Sigma$ defined in (\ref{ecurve}), and $X$ the total space of $E$. Then exactly one of the following conclusions holds.

\begin{enumerate}[label=(\Alph*)]
\item  $E=L_1 \oplus L_2$ for two line bundles $L_1$ and $L_2$ with $\operatorname{deg}(L_1)+\operatorname{deg}(L_2)=0$ and $L_1 \neq L_2$. We may choose $\xi_1$ and $\xi_2$ as local coordinates of $L_1$ and $L_2$ along the fiber directions respectively. Then any $f \in \mathcal{O}(X)$ has a Taylor series
\begin{equation}
f=\sum_{k=0}^{\infty} \sum_{p+q=k} a_{p, q} \xi_1^p \xi_2^q,    \label{taylor1}
\end{equation} 
where $a_{p, q}$ is the local coefficient of any element in $H^{0}(\Sigma, L_1^{-p} \otimes L_2^{-q})$ for any integers $p, q \geq 0$. 
\begin{comment}
In particular, there are two subcases.
\begin{enumerate} [label=(\Roman*)]
\item  If $L_1 \neq L_2$ and $L_1, L_2 \in \operatorname{Pic}^{0} (\Sigma)$, $a_{p, q}$ is nonzero (a constant) if and only if $L_1^{-p} \otimes L_2^{-q}=\mathcal{O}$.
\item  If $\operatorname{deg}(L_1)=-\operatorname{deg}(L_2)>0$, then $a_{p, q}=
     0$ if $p>q$. Moreover, $a_{p, p}=0$ if $L_2 \neq L_1^{-1}$, and it can be any constant if $L_2=L_1^{-1}$. If $p<q$, $a_{p, q}$ is determined by $H^{0}(\Sigma, L_1^{-p} \otimes L_2^{-q})$ with $\operatorname{dim} H^0(\Sigma, L_1^{-p} \otimes L_2^{-q})=(q-p)\operatorname{deg}(L_1)$. 
\end{enumerate}    
\end{comment}

\item  $E$ is associated with a representation defined by (\ref{gen-rep}). For any given $f \in \mathcal{O}(X)$, we consider its lift (still denoted by $f$ for simplicity) in $\mathcal{O}(\mathbb{C}^3)$. Then we have
  \begin{enumerate} [label=(\Roman*)]
      \item 
       If $\theta_1$ or $\theta_2$ is irrational, then $\mathcal{O}(X)=\{\text{constant}\}$. 
    \item
    If both $\theta_1$ and $\theta_2$ are rational, we may assume that $m$ is the smallest positive integer such that $m\theta_1$, $m\theta_2 \in \mathbb Z$. Then we consider two subcases.
    \begin{enumerate} [label=(\arabic*)]
        \item If $b_2 \ne b_1 \tau$, then any lift of $f \in \mathcal{O}(X)$ admits a Taylor series
        \begin{equation}
            f(z, z_1, z_2)=\sum_{k \geq 0, \, m|k}^{\infty} c_k z_2^{k},\ \text{where}\ c_k \in \mathbb{C}.  \label{taylor2}
        \end{equation} 
        \item If $b_2=b_1 \tau$, then any lift of $f \in \mathcal{O}(X)$ admits a Taylor series
        \begin{equation}
            f(z, z_1, z_2)=\sum_{k \geq 0,\,m|k}^{\infty}\, \sum_{p+q=k}  c_{p, q} (z_1-b_1 z z_2)^p z_2^q, \ \text{where}\ c_{p, q} \in \mathbb{C},\,\text{for integers}\ p, q \geq 0.   \label{taylor3}
        \end{equation}
    \end{enumerate}
    \end{enumerate}
    In Case (B), $E$ is decomposable (hence isomorphic to $L \oplus L$ for some $L \in \operatorname{Pic}^{0} (\Sigma)$ by Lemma \ref{isobundle}) if and only if $b_2=b_1 \tau$.
  
\end{enumerate}
\end{proposition}

\begin{proof}[Proof of Proposition \ref{Xfunction}]

If $E=L_1 \oplus L_2$, then any holomorphic function on its total space $X$ has a nice form of Taylor series. In our case, let $\xi_1$ and $\xi_2$ denote local coordinates of $L_1$ and $L_2$ along the fiber directions respectively. We may check that each of
\begin{align}
    \{a_{p,q} \xi_1^p \xi_2^q\ |\ p+q=n, \text{where}\  a_{p, q}\ \text{are local coefficients of some}\ \sigma \in H^{0}(\Sigma, L_1^{-p} \otimes L_2^{-q})\}
\end{align} is globally defined on $X$. In general, for any given $f \in \mathcal{O}(X)$ we have the Taylor series as in (\ref{taylor1}).

Regarding the coefficients $a_{p,q}$, we recall that for any line bundle $L$ over the elliptic curve $\Sigma$, $h^0(\Sigma,L)=1$ if $L\cong \mathcal{O}$, $h^0(\Sigma,L)=\deg L$ if $\deg L>0$, and $h^0(\Sigma,L)= 0$ otherwise.

Next we consider the case that $E$ is associated with a representation defined by (\ref{gen-rep}). For any given $f \in \mathcal{O}(X)$, we consider its lift $f \in \mathcal{O}(\mathbb{C}^3)$. For any $(z, z_1, z_2) \in \mathbb{C}^3$, $f$ satisfies
\begin{align}
    f(z,z_1,z_2)=f(z+1,e^{2\pi i \theta_1}(z_1+b_1 z_2), e^{2\pi i \theta_1}z_2)=f(z+\tau,e^{2\pi i \theta_2}(z_1+b_2 z_2), e^{2\pi i \theta_2}z_2).     \label{pi-inv}
\end{align}
We consider its Taylor expansion
\begin{equation}
    f(z,z_1,z_2)=f_0(z)+\sum_{k=1}^{+\infty} \sum_{t=0}^k f_{k-t,t}(z)z_1^{k-t} z_2^t.  \label{ftaylor}
\end{equation}
It follows from (\ref{pi-inv}) that
\begin{align*}
    f(z,z_1,z_2)&=f_0(z+1)+\sum_{k=1}^{+\infty} e^{2\pi i \theta_1 k} \sum_{t=0}^k f_{k-t,t}(z+1)(z_1+b_1z_2)^{k-t} z_2^t,\\
   f(z,z_1,z_2)&=f_0(z+\tau)+\sum_{k=1}^{+\infty} e^{2\pi i \theta_2 k} \sum_{t=0}^k f_{k-t,t}(z+\tau)(z_1+b_2 z_2)^{k-t} z_2^t.
\end{align*}
We proceed to compare the coefficients of the homogeneous monomials of $z_1$ and $z_2$. First of all, $f_0(z)$ must be constant since it is doubly periodic. For $k\ge 1$, 
\begin{equation}
        f_{k,0}(z)=e^{2\pi i \theta_1 k} f_{k,0}(z+1)=e^{2\pi i \theta_2 k} f_{k,0}(z+\tau).
\end{equation}
It follows that $f_{k,0}$ is bounded, hence constant. Whenever one of $e^{2\pi i \theta_1 k}$ and $e^{2\pi i \theta_2 k}$ differs from $1$, we have $f_{k,0}=0$. We may further show $f_{k-1,1}=\cdots=f_{0,k}=0$ by the same argument. 

If both $\theta_1$ and $\theta_2$ are rational, then we consider any positive integer $k$ so that $e^{2\pi i \theta_1 k}=e^{2\pi i \theta_2 k}=1$. Thus $f_{k,0}(z)=c_{k,0}$ which is a constant. Recall the binomial identity states that $(z_1+b_1 z_2)^k=\sum_{p=0}^k C^{p}_{k} b_1^{p} z_1^{k-p} z_2^{p}$. We consider the coefficient of $z_1^{k-1}z_2$ in (\ref{ftaylor}) and get
\[
f_{k-1,1}(z)=f_{k-1,1}(z+1)+c_{k,0}kb_1=f_{k-1,1}(z+\tau)+c_{k,0}kb_2.
\]
Then $f_{k-1,1}(z)$ is of linear growth and we may assume $f_{k-1,1}(z)=Az+B$. Then $A+c_{k,0}kb_1=A\tau+c_{k,0}kb_2=0$. It follows that $f_{k,0}=c_{k,0} \ne 0$ only if $b_2= b_1 \tau$.

To sum up, when $b_2 \ne b_1 \tau$ and $k \geq 1$, we have $f_{k,0}=0$ and $f_{k-1, 1}=B$. If $k \geq 2$, we may apply the above argument on $f_{k-2, 2}$. It follows that
\[
f_{k-2,2}(z)=f_{k-2,2}(z+1)+B(k-1)b_1=f_{k-2,2}(z+\tau)+B(k-1)b_2.
\]
Similarly, $f_{k-2,2}(z)$ is of linear growth. Hence $f_{k-1, 1}=0$ and $f_{k-2,2}(z)$=constant. In the end, we have
$f_{k,0}=f_{k-1,1}=\cdots=f_{1,k-1}=0$ and $f_{0,k}$ is constant. Hence $f$ only depends on $z_2$ and the Taylor series (\ref{taylor2}) holds.

When $b_2=b_1 \tau$ and $f_{k,0}(z) =c_{k,0} \neq 0$, we have
    \begin{align*}
        f_{k-1,1}(z)&=-c_{k,0} b_1kz+c_{k-1,1}.\\
        f_{k-2,2}(z)&=c_{k,0}(-1)^2\frac{k(k-1)}{2}b_1^2 z^2+c_{k-1,1}(k-1)(-b_1 z)+c_{k-2,2}. \\
        &\cdots \\
        f_{k-t,t}(z)&=\sum_{s=0}^t c_{k-s,s} C_{k-s}^{t-s}(-b_1 z)^{t-s}.
    \end{align*}
       Here $c_{k-t,t}$ are constants determined by $f_{k-t, t}(z)$. Hence
\begin{equation}
    \sum_{t=0}^k f_{k-t,t}(z)z_1^{k-t}z_2^t =\sum_{t=0}^k c_{k-t,t}(z_1 -b_1 zz_2)^{k-t}z_2^t.  \label{sum_re}
\end{equation}
To sum up, we have proved (\ref{taylor3}). Finally, we consider any vector bundle $E$ in Cases (B) and prove that it is decomposable if and only if $b_2=b_1 \tau$. Obviously, the `if' direction follows from Lemma \ref{isobundle}. Conversely, by Lemma \ref{isobundle} it suffices to show $E$ is indecomposable under the assumption of $b_1=0$ and $b_2 \neq 0$. We observe that any vector bundle $E$ in Cases (B) is isomorphic to $L \otimes \widehat{E}$ where $L \in \operatorname{Pic}^{0}(\Sigma)$ and $\widehat{E}$ is associated with the representation (\ref{gen-rep}) with parameters $\theta_1=\theta_2=b_1=0$ and $b_2 \neq 0$. By our discussion in Subsection \ref{discuss_E(2, 0)}, $\widehat{E}$ is indecomposable, so is $E=L \otimes \widehat{E}$.

\end{proof}

\begin{remark}\label{Iena_relation}
We work on invariant holomorphic functions to prove Proposition \ref{Xfunction} in an explicit way. There is a conceptually simpler and general approach to Proposition \ref{Xfunction} by considering $H^0\bigl(\Sigma,\operatorname{S}^k (E^{\ast})\bigr)$. Here $S^k(E^{\ast})$ is the $k$-th symmetric power vector bundle of the dual bundle $E^{\ast}$. Note that $\operatorname{S}^k (E^{\ast})$ can be solved if $E$ is of \ref{T3}; see \cite{Atiyah, Tu, Iena}. 
\end{remark}

Recall the interior of the fundamental domain $\mathcal{F}$ of modular group acting on the upper half plane $\mathbb{H}$.
\begin{align*}
\mathcal{F}=\{ z \in \mathbb{H} \ |\ |z|>1,\ |\operatorname{Re}z|<\frac{1}{2}\}.
\end{align*}
We also recall the following half-open fundamental domain after choosing one representative from each boundary identification.
\begin{equation}
\widetilde{\mathcal{F}}=\mathcal{F} \cup \{ z \in \mathbb{H} \ |\ |z|=1, -\frac{1}{2}<\operatorname{Re}z \leq 0, \ \text{or}\ |z| \geq 1, \operatorname{Re}z=-\frac{1}{2}\}. \label{funda}
\end{equation}

The following result studies the biholomorphic classification on total spaces of rank-two vector bundles of degree zero over an elliptic curve.

\begin{proposition}\label{biholoC}
    Let $\Sigma$ (resp.~$\widetilde{\Sigma}$) be the elliptic curve generated by $\{1, \tau\}$ (resp. $\{1, \widetilde{\tau}\}$), and $X$ (resp.~$\widetilde X$) the total space of a vector bundle $E$ (resp.~$\widetilde E$) of rank $2$ and degree $0$ over $\Sigma$ (resp.~$\widetilde{\Sigma}$). If there exists a biholomorphism $F: X \rightarrow \widetilde{X}$, then $\Sigma$ is biholomorphic to $\widetilde{\Sigma}$, and there exists a bundle isomorphism $\Phi: E \rightarrow \widetilde{E}$ which can be constructed from $F$ in an explicit way. Moreover, in the \ref{T3} case, we can determine all biholomorphisms between $X$ and $\widetilde{X}$.
\end{proposition}

\begin{proof}[Proof of Proposition \ref{biholoC}]
The key idea of the proof is to analyze the Taylor coefficients of the lifted biholomorphism. Any biholomorphism $F \colon  X \to \widetilde{X}$ lifts to an automorphism of $\mathbb{C}^3$, which, by a slight abuse of notation, we also denote by $F$. Set
\[
F(z,z_1,z_2)=(F_0(z,z_1,z_2),F_1(z,z_1,z_2),F_2(z,z_1,z_2)).
\]
We expand three components of $F$ into Taylor series with respect to $z_1$ and $z_2$.
\begin{align}
    &F_0(z,z_1,z_2)=\xi_0(z)+\xi_1(z)z_1+\xi_2(z)z_2+\sum_{i+j\ge2} \xi_{ij}(z)z_1^i z_2^j.  \label{defF1}\\
&F_1(z,z_1,z_2)=\lambda_0(z)+\lambda_1(z)z_1+\lambda_2(z)z_2+\sum_{i+j\ge2} \lambda_{ij}(z)z_1^i z_2^j. \label{defF2} \\
&F_2(z,z_1,z_2)=\eta_0(z)+\eta_1(z)z_1+\eta_2(z)z_2+\sum_{i+j\ge2} \eta_{ij}(z)z_1^i z_2^j. \label{defF3}
\end{align}
We divide the proof into the following steps. 

\vskip 0.1cm
\textbf{Step 1:} Show that $\Sigma$ is biholomorphic to $\widetilde{\Sigma}$. 

Note that $F$ is equivariant under the action of the deck transformation group. Comparing the terms on $\{z_1=z_2=0\}$, we obtain that there exist $p,q,r,s \in \mathbb{Z}$ such that 
\[
\xi_0(z+1)=\xi_0(z)+p+q\widetilde{\tau},\qquad
\xi_0(z+\tau)=\xi_0(z)+r+s\widetilde{\tau}.
\]
Then $\xi_0'(z)$ is constant and $\xi_0(z)=Az+B$ where $A=p+q\widetilde{\tau}$, $A\tau=r+s\widetilde{\tau}$, and $B$ is constant. Hence $A\Gamma \subset \widetilde{\Gamma}$. For any deck transformation $\gamma \in \Gamma$, it is direct to check that $F \circ \gamma \circ F^{-1} \in \widetilde{\Gamma}$. Similarly, $F^{-1} \circ \widetilde{\gamma} \circ F \in \Gamma$ when $\widetilde{\gamma} \in \widetilde{\Gamma}$. Thus for any $\widetilde{\gamma} \in \widetilde{\Gamma}$, there exists $k,l\in \mathbb{Z}$ such that $\widetilde{\gamma}=kA+lA\tau$. We conclude that $ps-qr=\pm 1$ by choosing $\widetilde{\gamma}=1, \widetilde{\tau}$. As we may set $\operatorname{Im}\tau, \operatorname{Im}\widetilde{\tau}>0$, we conclude that 
$ps-qr=1$ and $\Sigma \simeq \widetilde{\Sigma}$. In particular, if we assume $\tau,\, \widetilde{\tau} \in \mathcal{\widetilde{F}}$ defined in \eqref{funda}, $\tau=\widetilde{\tau}$.

\vskip 0.1cm
\textbf{Step 2:} Construct a biholomorphism $G$ that sends the zero section of $E$ to the zero section of $\widetilde{E}$.

Since $F(z,0,0)=(Az+B, \lambda_0(z),\eta_0(z))$, $\left( \lambda_0(\frac{w-B}{A}),\eta_0(\frac{w-B}{A})\right)$ defines a global section of $\widetilde{E}$. Take
\[
G(z,z_1,z_2)=\left( F_0, F_1-\lambda_0\left(\frac{F_0-B}{A} \right), F_2-\eta_0\left(\frac{F_0-B}{A} \right)  \right).
\]
Then $G$ is a well-defined biholomorphism while $G(z,0,0)=(Az+B,0,0)$.

\vskip 0.1cm
\textbf{Step 3:} Construct a bundle isomorphism $\Phi$ from $G$.

By computing the Jacobian of $G$ along $\{z_1=z_2=0\}$, one has $\lambda_1^G(z) \eta_2^G(z)- \lambda_2^G(z) \eta_1^G(z) \ne 0$. Here these coefficients refer to the Taylor expansion of $G$. Define
\begin{equation}\label{bundle_iso}
    \Phi(z,z_1,z_2)=(Az+B, \lambda_1^G(z)z_1+\lambda_2^G(z)z_2, \eta_1^G(z)z_1+\eta_2^G(z)z_2).
\end{equation}
Then $\Phi$ is the desired bundle isomorphism. 

\vskip 0.1cm
\textbf{Step 4:} Give an explicit expression for $\Phi$ case by case. Recall that any rank $2$ vector bundle of degree zero over an elliptic curve must belong to one of \ref{T1}, \ref{T2}, and \ref{T3}. For the \ref{T3} case, we determine all biholomorphisms $F$ on the corresponding total space.

\vskip 0.1cm

\textbf{\ref{T3} case:} In this case, $E$ (resp. $\widetilde E$) is associated with a representation (\ref{gen-rep}) where the parameters of $\rho$ (resp. $\widetilde{\rho}$) are $\theta_1, b_1, \theta_2$, and $b_2$ (resp. $\widetilde{\theta}_1, \widetilde{b}_1, \widetilde{\theta}_2$, and $\widetilde{b}_2$). According to Lemma \ref{isobundle}, we may assume that $b_1=\widetilde{b}_1=0$ after a bundle isomorphism. Note that $b_2$ (resp. $\widetilde{b}_2$) represents the previous $b_2-b_1 \tau$ (resp. $\widetilde{b}_2-\widetilde{b}_1\widetilde{\tau}$), which is nonzero for an indecomposable bundle. After we lift any biholomorphism $F$ between the total spaces $X$ and $\widetilde{X}$, there exist $p,q,r,s \in \mathbb{Z}$, such that 
    \begin{align}
        &F(z+1,\rho(\gamma_1)\begin{pmatrix}
            z_1 \\z_2
        \end{pmatrix})=(F_0(z,z_1,z_2)+p+q\widetilde{\tau}, (\widetilde{\rho}(\widetilde{\gamma_1}))^p (\widetilde{\rho}(\widetilde{\gamma_2}))^q \begin{pmatrix}
            F_1(z,z_1,z_2) \\ F_2(z,z_1,z_2)
        \end{pmatrix}). \label{correspond1}\\
        &F(z+\tau,\rho(\gamma_2)\begin{pmatrix}
            z_1 \\z_2
        \end{pmatrix})=(F_0(z,z_1,z_2)+r+s\tilde{\tau},  (\widetilde{\rho}(\widetilde{\gamma_1}))^r(\widetilde{\rho}(\widetilde{\gamma_2}))^s \begin{pmatrix}
            F_1(z,z_1,z_2) \\ F_2(z,z_1,z_2)
        \end{pmatrix}).\label{correspondtau}
    \end{align}

By Step 1, we may assume that $\tau=\widetilde{\tau}$, and set $A=p+q\tau$ and $A\tau=r+s\tau$. Next we rewrite all the components of equations $(\ref{correspond1})$ and $(\ref{correspondtau})$ as follows.
\begin{align}
    &F_0(z+1,e^{2\pi i \theta_1}z_1,e^{2\pi i \theta_1}z_2)=F_0(z,z_1,z_2)+A. \label{01} \\
    &F_0(z+\tau,e^{2\pi i \theta_2}(z_1+b_2 z_2),e^{2\pi i \theta_2}z_2)=F_0(z,z_1,z_2)+A\tau. \label{0tau} \\
    &F_1(z+1,e^{2\pi i \theta_1}z_1,e^{2\pi i \theta_1}z_2)=e^{2\pi i (p\widetilde{\theta}_1+q\widetilde{\theta}_2)}(F_1(z,z_1,z_2)+q\widetilde{b}_2F_2(z,z_1,z_2)).\label{11}\\
    &F_1(z+\tau,e^{2\pi i \theta_2}(z_1+b_2 z_2),e^{2\pi i \theta_2}z_2)=e^{2\pi i (r\widetilde{\theta}_1+s\widetilde{\theta}_2)}(F_1(z,z_1,z_2)+s\widetilde{b}_2F_2(z,z_1,z_2)).\label{1tau}\\
    &F_2(z+1,e^{2\pi i \theta_1}z_1,e^{2\pi i \theta_1}z_2)=e^{2\pi i (p\widetilde{\theta}_1+q\widetilde{\theta}_2)}F_2(z,z_1,z_2). \label{21} \\
    &F_2(z+\tau,e^{2\pi i \theta_2}(z_1+b_2 z_2),e^{2\pi i \theta_2}z_2)=e^{2\pi i (r\widetilde{\theta}_1+s\widetilde{\theta}_2)}F_2(z,z_1,z_2).\label{2tau}
\end{align}

We state a lemma which follows from the proof of Proposition \ref{Xfunction}.

\begin{lemma}\label{invarianth}
Let $\theta_i, \varphi_i \in [0, 1)$ for $i=1, 2$. Assume that $\operatorname{Im} \tau>0$ and $b_2 \ne0$. If a holomorphic function $h \in \mathcal{O}(\mathbb{C}^3)$ satisfies
\[
\begin{aligned}
&h(z+1,e^{2\pi i\theta_1}z_1,e^{2\pi i\theta_1}z_2)
=e^{2\pi i\varphi_1}h(z,z_1,z_2),\\
&h(z+\tau,e^{2\pi i\theta_2}(z_1+b_2z_2),
    e^{2\pi i\theta_2}z_2)
=e^{2\pi i\varphi_2}h(z,z_1,z_2),
\end{aligned}
\]
then $h$ only depends on $z_2$, i.e.,
\(
h(z,z_1,z_2)=\sum_{j=0}^{+\infty}c_j z_2^j.
\)
Moreover, if there exists $c_j \neq 0$, then 
\[
j\begin{pmatrix}
    \theta_1 \\ \theta_2
\end{pmatrix} -
\begin{pmatrix}
    \varphi_1\\ \varphi_2
\end{pmatrix} \in \mathbb{Z}^2.
\]

\end{lemma}

By \eqref{01} and \eqref{0tau}, $F_0(z,z_1,z_2)-Az$ satisfies Lemma \ref{invarianth}. Hence
\begin{equation}
    F_0(z,z_1,z_2)-Az=B+\xi_2 z_2 +\sum_{j \ge 2} \xi_{0j} z_2^j.
\end{equation}
Here all the coefficients are constants and
\[
\xi_2 \ne 0 \Rightarrow \theta_1=\theta_2=0, \qquad \xi_{0j} \ne 0 \Rightarrow j\theta_1, j\theta_2 \in \mathbb{Z}.
\]
Similarly, take $\varphi_1= p\widetilde{\theta}_1+q\widetilde{\theta}_2$ and $\varphi_2= r\widetilde{\theta}_1+s\widetilde{\theta}_2$. Then $F_2$ satisfies Lemma \ref{invarianth} and 
\[
F_2(z,z_1,z_2)=\eta_0+\eta_2 z_2 +\sum_{j\ge 2} \eta_{0j}z_2^j.
\]
We can apply the same analysis to $F^{-1}$, which yields that $F_2$ is an automorphism of $\mathbb{C}$. Thus $F_2=\eta_0+\eta_2 z_2$ with $\eta_2 \ne 0$. Hence if such a biholomorphism $F$ exists, then there exists $\begin{pmatrix}
             p&q\\r&s
          \end{pmatrix} \in SL(2, \mathbb{Z})$ such that
\begin{equation}
         \begin{pmatrix}
             p&q\\r&s
          \end{pmatrix}
         \begin{pmatrix}
        \widetilde{\theta}_1\\\widetilde{\theta}_2
        \end{pmatrix}-
         \begin{pmatrix}
            \theta_1 \\ \theta_2
        \end{pmatrix}\in \mathbb{Z}^2. \label{theta}
  \end{equation} 
 Moreover, $\eta_0=0$ when either $\theta_1$ or $\theta_2$ is non-zero. If $\theta_1=\theta_2=0$, it follows from \eqref{11} and \eqref{1tau} that  
\[ 
\lambda_0(z+1)-\lambda_0(z)=q\widetilde b_2\eta_0,\qquad \lambda_0(z+\tau)-\lambda_0(z)=s\widetilde b_2\eta_0. 
\]
Hence $\lambda_0'(z)$ is constant and $(s-q\tau)\widetilde b_2\eta_0=0$. Note that 
\[
(s-q\tau)(p+q\tau)=s(p+q\tau)-q\frac{r+s\tau}{p+q\tau} (p+q\tau)=ps+sq\tau-qr-qs\tau=1.
\]
It follows that $\eta_0=0$ and $F_2=\eta_2 z_2$.

Finally, we study $F_1$. It is direct to check that
\(
F_1-\frac{\eta_2\widetilde b_2}{Ab_2}z_1
-q\widetilde b_2\eta_2zz_2
\)
satisfies Lemma \ref{invarianth}. Hence $F_1$ can be determined in the same way. To sum up, we prove the following conclusion.

\begin{proposition}\label{typeIII_iso}
    Let $E$ and $\widetilde{E}$ belong to \ref{T3}. Assume that $\tau=\widetilde{\tau}$ and $b_1=\widetilde{b}_1=0$ without loss of generality. Then $X$ is biholomorphic to $\widetilde{X}$ if and only if there exist $\begin{pmatrix}
             p&q\\r&s
          \end{pmatrix} \in SL(2, \mathbb{Z})$ with $r+s\tau=\tau(p+q\tau)$ satisfying equation \eqref{theta}. If $F=(F_0,F_1,F_2)$ is a biholomorphism, then
    \begin{align*}
F_0&=Az+B+\xi_2z_2+\sum_{j\geq2}\xi_{0j} z_2^j,\\
 F_1&=\frac{\eta_2\widetilde b_2}{A b_2}z_1
       +q\widetilde b_2\eta_2 z z_2
       +\lambda_0+\lambda_2 z_2
       +\sum_{j\geq2}\lambda_{0j}z_2^j,\\
 F_2&=\eta_2 z_2.
    \end{align*}
    Here all the coefficients are constant, $A=p+q\tau$, $B, \lambda_2 \in \mathbb{C}$, $\eta_2 \in \mathbb{C}^{\ast}$, and the other coefficients satisfy
    \[
    (\xi_2, \lambda_0) \ne (0,0) \Rightarrow \theta_k=0, \qquad \xi_{0j} \ne 0 \Rightarrow j\theta_k \in \mathbb{Z}, \qquad \lambda_{0j} \ne 0 \Rightarrow (j-1)\theta_k \in \mathbb{Z} \qquad k=1,2.
    \]
    Moreover, any bundle isomorphism between $E$ and $\widetilde{E}$ can be determined as
    \begin{equation}
    \Phi= \left(Az+B,\frac{\eta_2\widetilde b_2}{A b_2}z_1
       +(q\widetilde b_2\eta_2 z +\lambda_2) z_2, \eta_2 z_2  \right).   \label{biholomap_new}
    \end{equation}
\end{proposition}

Next, we give a complete characterization of the bundle isomorphisms in the remaining two cases. We always assume $\widetilde\tau=\tau$. Every lifted bundle isomorphism is in the form of \eqref{bundle_iso} where $\lambda_1\eta_2-\lambda_2\eta_1$ is nowhere zero.

\vskip 0.1cm

\textbf{\ref{T1} Case:}  
In this case, $E=L_1 \oplus L_2$ with $L_1=L(0, \alpha)$ and $L_2=L(0, \beta)$, and $\widetilde{E}=\widetilde{L}_1 \oplus \widetilde{L}_2$ with $\widetilde{L}_1=L(0, \widetilde{\alpha})$ and $\widetilde{L}_2=L(0, \widetilde{\beta})$.

\begin{proposition}\label{typeI_iso}
    All the coefficients $\lambda_1,\lambda_2,\eta_1,\eta_2$ in \eqref{bundle_iso} are constants which satisfy the following equalities
    \[
    \lambda_1 \alpha(\gamma)=\widetilde{\alpha} (A\gamma) \lambda_1,\qquad
 \lambda_2 \beta(\gamma)= \widetilde{\alpha} (A\gamma) \lambda_2, \qquad
 \eta_1 \alpha(\gamma)= \widetilde{\beta}(A\gamma) \eta_1,\qquad
 \eta_2 \beta(\gamma)= \widetilde{\beta}(A\gamma) \eta_2.
    \]
    Moreover, $E$ is isomorphic to $\widetilde{E}$ if and only if 
    \(
 \{\alpha,\beta\}=\{\widetilde\alpha\circ A,\widetilde\beta\circ A\}
  \)
as a finite set for a given $A\in \mathbb{C}^*$ with $A\Gamma=\Gamma$.
\end{proposition}

\vskip 0.1cm

\textbf{\ref{T2} Case:} In this case, $E=L(H,\alpha)\oplus L(-H,\beta)$ and $\widetilde E=L(\widetilde H,\widetilde{\alpha})
\oplus L(-\widetilde H,\widetilde{\beta})$ with $H,\widetilde{H}>0$. Note that $H= \widetilde{H}$ when $E$ is isomorphic to $\widetilde{E}$. 
\begin{proposition}\label{typeII_iso}
    There exists a bundle isomorphism $\Phi$ such that $\Phi(z,0,0)=(Az+B,\cdot,\cdot)$ if and only if $H=\widetilde{H}$, $A\Gamma=\Gamma$ and
    \begin{equation}
        \frac{\alpha(\gamma)} {\widetilde{\alpha}(A\gamma)}
 =e^{\pi H(B\overline A\bar\gamma-A\overline B\gamma)},\qquad
 \frac{\beta(\gamma)}{\widetilde{\beta}(A\gamma)}
 =e^{-\pi H(B\overline A\bar\gamma-A\overline B\gamma)}.   \label{twocondi_new}
    \end{equation}
 Moreover, for any $A\in\mathbb{C}^*$ with $A\Gamma=\Gamma$ and $B$ satisfying \eqref{twocondi_new}, any bundle isomorphism between $E$ and $\widetilde{E}$ can be expressed as
\begin{equation}
 \Phi(z,z_1,z_2)=
 \left( Az+B,\lambda_1 e^{\pi HA\overline Bz}z_1+\lambda_2(z)z_2,
 \eta_2 e^{-\pi HA\overline Bz}z_2 \right),
\end{equation}
where $\lambda_1,\eta_2\in\mathbb{C}^*$ are constants, and $\lambda_2$ is any holomorphic function satisfying
\begin{equation}
 \lambda_2(z+\gamma)=
 \frac{\widetilde{\alpha}(A\gamma)}{\beta(\gamma)}
 e^{\,2\pi Hz\bar\gamma+\pi H|\gamma|^2+
\pi HB\overline A\bar\gamma}\lambda_2(z).
\end{equation}   
\end{proposition}
In particular, if $\lambda_1=\eta_2=1$ and $\lambda_2=0$, we obtain a special bundle isomorphism in the following.
\begin{equation}
    \Phi(z,z_1,z_2)=(A z+B,e^{\pi HA\overline{B}z} z_1, e^{-\pi HA\overline{B}z} z_2).  \label{Phidef_new}
\end{equation}
We will use \eqref{Phidef_new} to study Hermitian geometry of total spaces in the proof of Theorem \ref{Hermipoly}.

\end{proof}

\section{Hermitian geometry on total spaces of rank \texorpdfstring{$2$}{TEXT} bundles of degree zero}\label{sec5}

Let $E$ be a rank $2$ bundle of degree zero over an elliptic curve and $X$ its total space. In this section we construct complete Hermitian metrics on $X$ and prove Theorem \ref{Hermipoly}. If $E$ is of \ref{T1}, the standard Euclidean metric descends to $X$ as a complete flat K\"ahler metric $g$. It remains to construct the desired Hermitian metric $g$ if $E$ is of \ref{T2} or \ref{T3} respectively.

\subsection{The Type II case} In this subsection, we prove Theorem \ref{Hermipoly} when $E$ is of \ref{T2}, i.e. $E=L_1 \oplus L_2$ where $\operatorname{deg}L_1=-\operatorname{deg}L_2>0$. First of all, let $X$ be the total space of a line bundle of the form $L(H,\alpha)$ and introduce
\begin{equation}
    \omega_{L}=\sqrt{-1} (dz \wedge d\bar z+e^{\pi H |z|^2} \partial (e^{-\pi H |z|^2} \xi) \wedge \bar \partial (e^{-\pi H |z|^2}\overline{\xi})).   \label{HmetricT2}
\end{equation}

\begin{lemma}\label{complete2}
$\omega_{L}$ defined in (\ref{HmetricT2}) is a complete Hermitian metric on $X$.     
\end{lemma}
\begin{proof}[Proof of Lemma \ref{complete2}]
It is direct to check that $\omega$ is a well-defined Hermitian metric on $X$ with its local components
\begin{equation}
       \begin{pmatrix}
        1+e^{-\pi H |z|^2} (\pi H)^2 |z|^2 |\xi|^2  & e^{-\pi H |z|^2} (-\pi H \bar{z})\xi \\
        e^{-\pi H |z|^2} (-\pi H z)\bar{\xi} & e^{-\pi H |z|^2}
    \end{pmatrix}.  \label{HmetricT2f}
\end{equation}
Indeed, we may rewrite (\ref{HmetricT2f}) as
\begin{equation*}
    g=\begin{pmatrix}
        1&0 \\ 0&0
    \end{pmatrix} +c
    \begin{pmatrix}
        |a|^2 & \bar a \\
        a &1
    \end{pmatrix},\ \ \text{where}\ c=e^{-\pi H|z|^2} \text{and}\ a=-\pi H z \overline{\xi}.
\end{equation*}
We consider $\omega$ as the lifted metric on the universal cover of $X$. If $H=0$, $\omega$ is the Euclidean metric. We consider the case $H \neq 0$. For a fixed point $p=(0, 0)$,
let $\gamma(t)=(z(t), \xi(t))$ for $t \in [t_0, \infty)$ denote any $C^1$ curve from $p$. Then 
$g(\gamma'(t), \gamma'(t))=|z'(t)|^2 +c|\overline a z'(t)+\xi'(t)|^2$. Assume that the length of $\gamma$ is $L<\infty$.
As $\int_{t_0}^{t_1} \sqrt{g(\gamma'(t), \gamma'(t))} dt \geq \int_{t_0}^{t_1} |z'(t)|dt$, we get that $\sup_{t \in [t_1, t_2]}|z(t)| \leq L$.
\begin{align*}
    g(\gamma'(t), \gamma'(t)) &=(1+c|a|^2) |z'(t)|^2+2\operatorname{Re}(c\bar a z'(t) \overline{\xi'(t)})+c|\xi'(t)|^2\\
    &=(1+c|a|^2) \Big|z'(t)+\frac{c a\xi'(t)}{1+c|a|^2}\Big|^2+\frac{c |\xi'(t)|^2}{1+c|a|^2}\\
    &\ge \frac{|\xi'(t)|^2}{e^{\pi H|z(t)|^2}+ (\pi H) ^2|z(t)|^2|\xi(t)|^2} 
\end{align*}
Let $C_1=\max(e^{\pi H L^2}, 1)$. It follows that
\begin{equation*}
    L \geq \int_{t_0}^{t_1} \sqrt{g(\gamma'(t), \gamma'(t))}\,dt \geq \int_{t_0}^{t_1} \frac{|\xi'(t)|}{\sqrt{C_1+(\pi H)^2 L^2 |\xi(t)|^2}} \,dt \geq  \int_{t_0}^{t_1} \frac{\frac{d}{dt}|\xi(t)|}{\sqrt{C_1+(\pi H)^2 L^2 |\xi(t)|^2}}dt.
\end{equation*}
For any interval $I \subset [t_0, \infty)$ such that the corresponding part of $\xi(t)$ lies outside $|\xi| \leq \frac{C_1}{\pi |H|}$, we could show $\sup_{t \in I} |\xi|(t) \leq \frac{C_1}{\pi |H|} e^{\sqrt{2}\pi |H|L}$.
It follows that $\sup_{t \in [t_1, \infty)}|\xi(t)| \leq C_2$. Therefore, we have shown that the lifting metric $\omega$ on the universal cover of $X$ (hence $\omega$ on $X$) is complete.
\end{proof}

Motivated by (\ref{HmetricT2}), we introduce a Hermitian metric on the total space $X$ of $E=L_1(H,\alpha) \oplus L_2(-H,\beta)$ as follows.
\begin{align}
    \omega=&\sqrt{-1}\Big(dz \wedge d\bar z+e^{\pi H |z|^2} \partial (e^{-\pi H |z|^2}\xi_1) \wedge \bar \partial (e^{-\pi H |z|^2}\overline{\xi_1})   \label{HmetricT2g} \\
    &+e^{-\pi H |z|^2} \partial (e^{\pi H |z|^2}\xi_2) \wedge \bar \partial (e^{\pi H |z|^2}\overline{\xi_2})\Big).    \nonumber
\end{align}

\begin{lemma}\label{complete2g}
$\omega$ defined in (\ref{HmetricT2g}) is a complete Gauduchon metric on $X$ with zero Chern-Ricci curvature.  
\end{lemma}

\begin{proof}[Proof of Lemma \ref{complete2g}]
First of all, we may check that $\omega$ is invariant under the action
\[
(z, \xi_1, \xi_2) \rightarrow  (z+\gamma, \alpha(\gamma) e^{\pi H z \overline{\gamma}+\frac{\pi}{2}H |\gamma|^2} \xi_1, \beta(\gamma) e^{-\pi H z \overline{\gamma}-\frac{\pi}{2}H |\gamma|^2} \xi_2).
\]
Hence $\omega$ descends onto $X$ and its local components read
\begin{equation}
     \Scale[0.92]{
       \begin{pmatrix}
        1+e^{-\pi H |z|^2} (\pi H)^2 |z|^2 |\xi_1|^2+e^{\pi H |z|^2} (\pi H)^2 |z|^2 |\xi_2|^2  & e^{-\pi H |z|^2} (-\pi H \overline{z})\xi_1  &  
        e^{\pi H |z|^2} (\pi H \overline{z})\xi_2\\
        e^{-\pi H |z|^2} (-\pi H z)\overline{\xi_1} & e^{-\pi H |z|^2}  &   0\\
        e^{\pi H |z|^2} (\pi H z)\overline{\xi_2}   &  0  &  e^{\pi H |z|^2}
    \end{pmatrix}.  \label{HmetricT2gf}
    }
\end{equation}
After a comparison between (\ref{HmetricT2f}) and (\ref{HmetricT2gf}), we see that the completeness of $\omega$ can be proved similarly as in Lemma \ref{complete2}. It follows from Lemma \ref{detlemma} or a direct computation that $\det(g)=1$. Then the first Chern-Ricci curvature is equal to zero. It is straightforward to check that $\omega$ is Gauduchon. Namely, $\sqrt{-1}\partial \bar \partial (\omega^2)=0$ in the sense of \cite{Gauduchon1977}. 
\begin{comment}
 Note that
\begin{equation*}
    \partial \bar \partial \omega^2 = \partial \bar \partial ( \partial (e^{-\pi H |z|^2}\xi_1) \wedge \bar \partial (e^{-\pi H |z|^2} \overline{\xi_1}) \wedge \partial (e^{\pi H |z|^2}\xi_2) \wedge \bar \partial (e^{\pi H |z|^2}\overline{\xi_2}))
\end{equation*}
  Denote $a=e^{-\pi H |z|^2}\xi_1$ and $b=e^{\pi H |z|^2}\xi_2$.
\begin{equation*}
    \partial \bar \partial (\omega^2)=\partial (\bar \partial \partial a \wedge \bar \partial \bar a \wedge \partial b \wedge \bar \partial \bar b 
    + \partial a \wedge \bar \partial \bar a \wedge \bar \partial \partial b \wedge \bar \partial \bar b).
\end{equation*}
Then
\begin{align*}
    &\partial(\bar \partial \partial a \wedge \bar \partial \bar a \wedge \partial b \wedge \bar \partial \bar b) \\
    =&\bar \partial \partial a \wedge \partial \bar \partial \bar a \wedge \partial b \wedge \bar \partial \bar b
    + \bar \partial \partial a \wedge \bar \partial \bar a \wedge \partial b \wedge \partial \bar \partial \bar b \\
    =&\pi^2 H^2 d\bar z \wedge d\xi_1 \wedge dz \wedge d\overline{\xi_1} \wedge d\xi_2 \wedge d\overline{\xi_2}+(-\pi^2 H^2) d\bar z \wedge d\xi_1  \wedge d\overline{\xi_1} \wedge d\xi_2 \wedge dz \wedge d\overline{\xi_2} \\
    =&0.
\end{align*}
Similarly, we may show $\partial (\partial a \wedge \bar \partial \bar a \wedge \bar \partial \partial b \wedge \bar \partial \bar b)=0$. Therefore $\omega$ is Gauduchon.  
\end{comment}
\end{proof}

\begin{proof}[Proof of Theorem \ref{Hermipoly} when $E$ is of \ref{T2}]

In view of Lemma \ref{complete2g}, it suffices to determine $\mathcal{O}_d(X, g)$ which is defined in (\ref{polyspaced}). Consider the complete Hermitian metric $(X, g)$ defined in (\ref{HmetricT2g}). Let $\Sigma$ denote the zero section of the vector bundle $E=L_1(H,\alpha) \oplus L_2(-H,\beta)$, and $\widetilde{d}(q)$ the distance function (with respect to $g$) of any point $q \in X$ to $\Sigma$. Equivalently, we need to characterize all $f \in \mathcal{O}(X)$ so that there exists some $C(d, f)>0$
\begin{align}
    |f(q)| \leq C(\widetilde{d}(q)+1)^{d}\ \text{for all}\ q \in X.   \label{polyspaced2}
\end{align}

With a slight abuse of notation, let $\pi: E \rightarrow \Sigma$ be the projection and $\xi_1$ and $\xi_2$ denote local coordinates of $L_1$ and $L_2$ along the fiber directions respectively. It follows from (\ref{taylor1}) that any $f \in \mathcal{O}(X)$ can be expressed as
$f=\sum_{n=0}^{\infty} \sum_{n=p+q} a_{p, q} \xi_1^p \xi_2^q$, where $a_{p, q}$ is the local coefficient of any element in $H^{0}(\Sigma, L_1^{-p} \otimes L_2^{-q})$ for any integers $p, q \geq 0$. In order to characterize (\ref{polyspaced2}), it suffices to show both $\xi_1$ and $\xi_2$ grow linearly with respect to $\widetilde{d}$ on $\pi^{-1}(U)$ for a small open set $U \subset \Sigma$. We argue by three steps.

\textbf{Step 1:} We observe that a special fiber $\pi^{-1}(0)$ of the point $z=0$ on $\Sigma$ is totally geodesic with respect to the Riemannian connection of $g$. Indeed, by (\ref{HmetricT2gf}), $\frac{\partial}{\partial z}$ is the normal direction along $\pi^{-1}(0)$. The Riemannian Christoffel symbols (in terms of real coordinates) satisfy $\Gamma_{ab}^{c}=0$ where $a, b$ is from the real and imaginary coordinates of $\xi_1, \xi_2$, and $c$ is from those of $z$.

\textbf{Step 2:} We show the fiber $\pi^{-1}(B)$ for any point $z=B$ is totally geodesic. We consider $\Phi \in \operatorname{Aut}(\mathbb{C}^3)$ defined in \eqref{Phidef_new} in the case of $A=1$. Namely
\begin{equation}
    \Phi(z, \xi_1, \xi_2)=(z+B, e^{\pi H \overline{B} z} \xi_1, e^{-\pi H \overline{B} z} \xi_2).
    \label{PhidefN}
\end{equation} By \eqref{twocondi_new}, $\Phi$ descends to a bundle isomorphism if $e^{\pi H (\overline{B}\gamma-B\overline{\gamma})}=1$. Pick an arbitrary point $z=B$. It suffices to work with the lifted metric $\omega$ on the universal cover and show that the preimage of each fiber $\pi^{-1}(B)$ is totally geodesic.

It is easy to check the pull back metric $\Phi^{\ast}\omega$ has the corresponding local components 
\begin{equation}
  \begin{pmatrix}
        1+e^{-\pi H |z|^2-\pi H |B|^2} (\pi H)^2 |z|^2 |\xi_1|^2+e^{\pi H |z|^2+\pi H |B|^2} (\pi H)^2 |z|^2 |\xi_2|^2  &  A\\
       B &     C
    \end{pmatrix}.
    \end{equation}
Here we set
  \begin{align*}
      A&=\begin{pmatrix}
          e^{-\pi H |z|^2-\pi H |B|^2} (-\pi H \overline{z})\xi_1  &  
        e^{\pi H |z|^2+\pi H |B|^2} (\pi H \overline{z})\xi_2 
      \end{pmatrix},\\
      B&=\begin{pmatrix}
          e^{-\pi H |z|^2-\pi H |B|^2} (-\pi H z)\overline{\xi_1} \\  e^{\pi H |z|^2+\pi H |B|^2} (\pi H z)\overline{\xi_2}
      \end{pmatrix},\ \ 
      C=\begin{pmatrix}
          e^{-\pi H |z|^2-\pi H |B|^2}   &  0 \\
          0  &  e^{\pi H |z|^2+\pi H |B|^2}
      \end{pmatrix}.
  \end{align*}
  By a similar argument as in Step 1, we get that the fiber of $z=0$ is totally geodesic with respect to $\Phi^{\ast}\omega$. It implies that the fiber of $z=B$ is totally geodesic with respect to $\omega$. Equivalently, it amounts to finding a coordinate change near the fiber 
  $\pi^{-1}(B)$ so that $\omega$ is reduced to a specific diagonal form as that near the fiber $\pi^{-1}(0)$.

\textbf{Step 3:} We estimate the growth of $\xi_i\,(i=1, 2)$ with respect to $\widetilde{d}$.
Now each fiber $\pi^{-1}(z)$ is totally geodesic. Moreover, any geodesic of $X$ which begins with $z \in \Sigma$ and meets $\Sigma$ orthogonally must stay inside the fiber $\pi^{-1}(z)$. In other words, for any $q \in X$, $\widetilde{d}$ is realized by a geodesic inside the fiber containing $q$. Therefore, we are in a similar situation as in the Calabi method (Theorem \ref{vecmetric}). Since the induced metric on each fiber is Euclidean, we conclude that both $\xi_1$ and $\xi_2$ grow linearly with respect to $\widetilde{d}$. The corresponding statement on the $\mathcal{O}_d(X, g)$ follows similarly as in the proof of Proposition \ref{holoCala}.

\end{proof}

\subsection{The Type III case} Assume that $E$ is an indecomposable rank $2$ bundle of degree zero which is associated with a representation defined in (\ref{gen-rep}). Now we introduce a Hermitian metric on the total space of $E$. We begin with the following $(1, 1)$ form on $\mathbb{C}^3$. Recall that $z=x+y\sqrt{-1}$ and $\Sigma$ is the elliptic curve defined in (\ref{ecurve}).
\begin{align}
    \omega=&\sqrt{-1}\Big[dz \wedge d\bar z +\partial\Big(z_1-(b_1 x+\frac{b_2-b_1 \operatorname{Re}\tau}{\operatorname{Im}\tau}y)z_2\Big)\wedge \overline \partial\Big(\overline {z_1}-(\overline {b}_1 x+\frac{\overline {b}_2-\overline {b}_1 \operatorname{Re}\tau}{\operatorname{Im}\tau}y)\overline {z_2}\Big)   \label{Gaugen}\\
    &+dz_2 \wedge d\overline {z_2}\Big].    \nonumber
\end{align}

Now we review some special cases of (\ref{Gaugen}). If the representation (\ref{gen-rep}) associated to $E$ satisfies $b_2=b_1 \tau$, then (\ref{Gaugen}) is reduced to
\[
\omega=\sqrt{-1}\Big(dz \wedge d\overline{z}+\partial(z_1-b_1zz_2) \wedge \bar{\partial}(\overline{z_1}-\overline{b_1zz_2})+dz_2 \wedge d\overline{z_2}\Big).
\]
By Lemma \ref{isobundle}, we may choose $F(z, z_1, z_2)=(z, z_1+b_1 z z_2, z_2) \in \operatorname{Aut}(\mathbb{C}^3)$ which descends to a bundle isomorphism between $L \oplus L$ and $E$ for some $L \in \operatorname{Pic}^{0} (\Sigma)$. Then $F^{\ast}\omega$ becomes the descended Euclidean metric on $X$.

If we choose $E$ as (\ref{Edef}) and the corresponding $\tau=\sqrt{-1}$, then 
(\ref{Gaugen}) is reduced to
\[
\omega= \sqrt{-1}(dz \wedge d\bar z +\partial(z_1-yz_2)\wedge \bar{\partial}(\overline{z_1}-y\overline{z_2})+dz_2 \wedge d\overline{z_2}).
\] We note that both $z_1-yz_2$ and $z_2$ terms appear in Paun's Hermitian fiber metric defined in (\ref{Hdef}). In this sense, (\ref{Gaugen}) is a natural `Hermitian' generalization of Calabi's K\"ahler metric defined in (\ref{Cmetric}). By a direct calculation, we may show the following result.
\begin{lemma}\label{metriclist}
    $\omega$ defined by (\ref{Gaugen}) descends to a Hermitian metric on $X$ such that
    \begin{enumerate}[label=(\arabic*)]
        \item $\omega$ is  Gauduchon.
        \item The first Chern-Ricci curvature of $\omega$ is zero.
        \item $\omega$ is a K\"ahler metric if and only if $b_2=b_1 \tau$.
    \end{enumerate}
\end{lemma}

Now we focus on the case $b_2 \neq b_1 \tau$ which corresponds to all indecomposable bundles by Proposition \ref{Xfunction}. It suffices to consider any fixed pair of $(b_1, b_2)$ with $b_2 \neq b_1 \tau$ as the resulting total space with the same value of $(\theta_1, \theta_2)$ must be biholomorphic. After we choose $b_1=1$ and $b_2 =\overline{\tau}$, (\ref{Gaugen}) is further reduced to
\begin{align}
    \omega=\sqrt{-1} \Big[dz \wedge d\bar z +dz_1 \wedge d\overline {z_1} -z dz_1 \wedge d\overline {z_2} -\bar z dz_2 \wedge d\overline {z_1} + (1+|z|^2) dz_2 \wedge d\overline {z_2} \Big].  \label{Gaugenspe}
\end{align}

\begin{lemma}\label{repcomp}
$\omega$ defined in (\ref{Gaugenspe}) is complete.
\end{lemma}

\begin{proof}[Proof of Lemma \ref{repcomp}]
We may follow the proof of Lemma \ref{complete2}. Consider the lifting metric $\omega$ on the universal covering $\widetilde{X}$. For a fixed point $p=(0, 0, 0)$, let 
$\gamma(t)=(z(t), \lambda(t), \eta(t))$ be any $C^1$ curve from $p$ and any point $q=(z_{\ast}, z_{1\ast}, z_{2\ast})$. Once we write $\gamma'(t)=(z'(t), \lambda'(t), \eta'(t))$, then
\begin{equation}
    g(\gamma'(t), \gamma'(t))=|z'(t)|^2+|\lambda'(t)-\overline{z} \eta'(t)|^2+|\eta'(t)|^2.  \label{goodlower}
\end{equation}
If either $|z_{\ast}|$ or $|z_{2\ast}|$ goes to $\infty$, we may estimate $L(\gamma(t)) \rightarrow \infty$. Otherwise, we rewrite the above equation as
\begin{equation*}
    g(\gamma'(t), \gamma'(t))=|z'(t)|^2+(1+|z|^2)\Big| \frac{\lambda'(t)z}{1+|z|^2}-\eta'(t) \Big|^2+\frac{|\lambda'(t)|^2}{1+|z|^2}.
\end{equation*} The remaining argument follows exactly as in the proof of Lemma \ref{complete2}. 

\end{proof}

\begin{proof}[Proof of Theorem \ref{Hermipoly} when $E$ is of \ref{T3}]

As $E$ is of \ref{T3}, we have $b_2 \neq b_1\tau$ in the representation (\ref{gen-rep}). By Proposition \ref{Xfunction}, $\mathcal{O}(X)$ contains a nonconstant function if and only if both $\theta_1$ and $\theta_2$ are rational. Let $m$ be the smallest positive integer such that $m\theta_1$, $m\theta_2 \in \mathbb Z$. By (\ref{taylor2}), we need to estimate the growth of $z_2^m$ with respect to $d_g$ on $X$.

Let $\rho: \widetilde{X} \rightarrow X$ denote the universal covering and $\pi: E \rightarrow \Sigma$ the projection map. Unlike the \ref{T2} case, the fiber $\pi^{-1}(0)$ for the point $z=0$ on $\Sigma$ is no longer totally geodesic. Fix $p=\rho((0, 0, 0))$ and let $q=\rho((z, z_1, z_2)) \in X$. Let $d_g(q)$ denote the distance from $p$ to $q$. Thanks to (\ref{goodlower}), we have $d_g(q) \geq |z_2(q)|$ on $\widetilde{X}$. It descends to $X$ as 
$|z_2^m(q)| \leq (d_g(q))^m$. Moreover, it attains an equality along the sequence $\rho((0, 0, z_2))$ as $|z_2| \rightarrow \infty$. Hence, the desired characterization of $\mathcal{O}_d(X, g)$ follows as in the proof of Proposition \ref{holoCala}.

\end{proof}

\begin{remark}\label{quasirem}
    We may view $\omega$ defined in \eqref{Gaugen} as a family of Hermitian metrics on the same total space $X$. Moreover, they are quasi-isometric to each other. First of all, we may check the bundle isomorphism defined in Lemma \ref{isobundle} is an isometry with respect to $\omega$. Therefore, it suffices to consider a vector bundle $E$ of \ref{T3} with $b_1=0$ and $b_2 \neq 0$. As $b_2$ varies, it is direct to check the bundle isomorphism defined in \eqref{biholomap_new} is a quasi-isometry.   
\end{remark}

\section{Examples of function theory on total spaces of vector bundles of nonzero degree}
\label{sec6}

In this section, we consider a class of vector bundles of nonzero degrees over an elliptic curve, and prove Proposition \ref{n2d-1noholo}. 
\begin{comment}
    It shows that there are no nonconstant holomorphic functions on the total spaces.
\end{comment} 

Following Grothendieck's exposition of Weil's result \cite{Grothendieck}, Narasimhan-Seshadri \cite[Proposition~6.2, p.~551]{NS1965} proved that any indecomposable vector bundle of rank $n$ and degree $d$ with $-n<d \leq 0$ over a compact Riemann surface $M$ of genus $\geq 1$ can be associated to a representation of some Fuchsian group constructed from $\pi_1(M)$. We begin with a lemma regarding the existence of a hyperbolic orbifold structure on any compact Riemann surface with genus $\geq 1$.

\begin{lemma}[see, for example, {\cite[p.~549]{NS1965}}]\label{orbifold1}
    Let $M$ be any compact Riemann surface with genus $\geq 1$. Fix a point $x_0 \in M$, and an integer $N \geq 2$. There exists a simply-connected Riemann surface $U$ (which is biholomorphic to the upper half plane $\mathbb{H}$), and a branched covering $p: U \rightarrow M$ with the following properties:
    \begin{enumerate}
        \item  The branch locus on $M$ is the single point $x_0$. 
        \item  Let $\widetilde{\pi}$ be the covering transformation group of $p$. For any $y \in p^{-1}(x_0)$, the stabilizer subgroup of $\widetilde{\pi}$ at $y$ which is denoted by $\pi_y$, is a finite cyclic group of order $N$. Moreover, $N$ is independent of the choice of $y \in p^{-1}(x_0)$.
    \end{enumerate}
\end{lemma}

Given an elliptic curve $\Sigma$ with a fixed point $x_0 \in \Sigma$, and an integer $N \geq 2$. By Lemma \ref{orbifold1}, we have a branched covering $p: \mathbb{H} \rightarrow \Sigma$. The corresponding covering transformation group $\widetilde{\pi}$ is also called the orbifold fundamental group of $(\Sigma, x_0)$. It is an extension of $\pi_1(\Sigma)$ by 
$\widetilde{\pi}_y \coloneqq <g \pi_y g^{-1}\ |\ g \in \widetilde{\pi}> $, which is the normal closure of $\pi_y$ in $\widetilde{\pi}$, as follows
\begin{equation}
\begin{tikzcd}[column sep=20pt,row sep=20pt]
1 \arrow{r} & \widetilde{\pi}_y  \arrow{r}{} & \widetilde{\pi}  \arrow{r}{} &  \pi_1(\Sigma) \arrow{r} & 1.
\end{tikzcd}  \label{groupext}
\end{equation} Moreover, $\widetilde{\pi}$ has a presentation
\begin{align}
    \widetilde{\pi}=\{<\gamma_1, \gamma_2, \gamma_0>\ |\ [\gamma_1, \gamma_2]\gamma_0=\gamma_0^{N}=1\ \},  \label{piwpresent}
\end{align}
and can be realized as a Fuchsian group of signature $(1; N)$. In other words, there exists a representation $\sigma: \widetilde{\pi} \rightarrow \operatorname{PSL}(2, \mathbb{R})$ whose image is a Fuchsian group of signature $(1; N)$. We refer to \cite{Beardon, Katok} for basic facts on Fuchsian groups. We write the left action of $\widetilde{\pi}$ on $z \in \mathbb{H}$ as $z \rightarrow \sigma(\gamma)(z)$ for any $\gamma \in \widetilde{\pi}$. Recall that a character $\tau$ of $\pi_y$ is defined as an element in $\operatorname{Hom}(\pi_{y}, \mathbb{C}^{\ast})$. Following \cite[p.~549]{NS1965}, we call a representation $\rho: \widetilde{\pi} \rightarrow \operatorname{GL}(n, \mathbb{C})$  \emph{of type $\tau$} if $\rho(\gamma)=\tau(\gamma)E$ for any $\gamma \in \pi_y$ and $E$ the identity matrix. Given any representation $\rho$ of type $\tau$, we define an action of $\widetilde{\pi}$ on $\widetilde{E}=\mathbb{H} \times \mathbb{C}^n$ by
\begin{equation}
  (z, z_1, \cdots, z_n) \rightarrow  (\sigma(\gamma)(z), \rho(\gamma)(z_1, \cdots, z_n)),\ \ \ \gamma \in \widetilde{\pi}.    \label{actionpic}
\end{equation}
According to \cite[p.~549]{NS1965}, we call such a rank $n$ trivial bundle $\widetilde{E}$ with the action (\ref{actionpic}) a \emph{$\pi$-bundle}. Let $p_{\ast}(\mathcal{O}(\widetilde{E}))$ denote the direct image sheaf of $\mathcal{O}(\widetilde{E})$ of $p: \mathbb{H} \rightarrow \Sigma$, and $p_{\ast}^{\widetilde{\pi}} (\mathcal{O}(\widetilde{E}))$ the subsheaf of $\widetilde{\pi}$-invariant sections. It can be shown that $p_{\ast}^{\widetilde{\pi}} (\mathcal{O}(\widetilde{E}))$ is a locally free sheaf of rank $n$. Let $p_{\ast}^{\widetilde{\pi}} (\widetilde{E})$ denote the corresponding vector bundle on $\Sigma$.

\begin{lemma}[{\cite[Proposition~6.2]{NS1965}}]\label{degnonzero}
Let $n \geq 2, d$ be two integers with $-n < d \leq -1$. Consider any fixed point $x_0$ on the elliptic curve $\Sigma$ defined in (\ref{ecurve}). Let $p: \mathbb{H} \rightarrow \Sigma$ be the branched covering in Lemma \ref{orbifold1} with $N=n$. Assume that $\gamma_0$ is the generator of the cyclic group $\pi_y$ ($y \in p^{-1}(x_0)$) and choose the character $\tau(\gamma_0)=e^{-\frac{2\pi i d}{n}}$. Then for any given indecomposable vector bundle $E$ of rank $n$ and degree $d$, there exists a representation $\rho: \widetilde{\pi} \rightarrow \operatorname{GL}(n, \mathbb{C})$ of type $\tau$ such that $E$ is isomorphic to $p_{\ast}^{\widetilde{\pi}} (\widetilde{E})$.
\end{lemma}

Let $N=n=2$ and $d=-1$ in Lemma \ref{orbifold1} and Lemma \ref{degnonzero}. We define \textit{a pointed elliptic curve} as an elliptic curve $\Sigma$ with a choice of $x_0 \in \Sigma$. There exists a pointed elliptic curve whose corresponding covering transformation group $\widetilde{\pi}$ (in Lemma \ref{orbifold1}) is an arithmetic Fuchsian group of signature $(1; 2)$. The particular example we discussed is associated to a representation $\sigma: \widetilde{\pi} \rightarrow \operatorname{PSL}(2, \mathbb{R})$
\begin{equation}
\sigma(\gamma_1)=\alpha=\begin{pmatrix}
    \frac{\sqrt{6}+\sqrt{2}}{2} & 0 \\
    0  &   \frac{\sqrt{6}-\sqrt{2}}{2}
\end{pmatrix},\ \ \ \sigma(\gamma_2)=\beta=\begin{pmatrix}
    \sqrt{2} & 1 \\
    1  &   \sqrt{2}
\end{pmatrix}.  \label{defrepsig}   
\end{equation}
It follows from (\ref{defrepsig}) that
\[
\sigma(\gamma_0)=\delta=\begin{pmatrix}
    \sqrt{3} & -(\sqrt{2}+\sqrt{6}) \\
   \sqrt{6}-\sqrt{2}  &   -\sqrt{3}
\end{pmatrix},\ \text{and} \ \ \delta^2=[-E]=\operatorname{Id} \in \operatorname{PSL}(2, \mathbb{R}). 
\]
As $\delta$ has a unique fixed point $z_0=\frac{\sqrt{3}+i}{\sqrt{6}-\sqrt{2}} \in \mathbb{H}$, 
the map $p: \mathbb{H} \rightarrow \Sigma$ in Lemma \ref{orbifold1} satisfies $p(z_0)=x_0 \in \Sigma$. Indeed, this particular choice of $\widetilde{\pi}$ is from the list of the case $e=2$ in \cite[Theorem~4.1, p.~392]{Takeuchi}.

Now we define a representation $\rho: \widetilde{\pi} \rightarrow U(2)$ of type $\tau$ as follows
\begin{equation}
    \rho(\gamma_1)=A=e^{2\pi i\theta_1} \begin{pmatrix}
    0 &  1 \\
    1  &  0
\end{pmatrix},\ \ \rho(\gamma_2)=B=e^{2\pi i\theta_2} \begin{pmatrix}
    -1 &  0 \\
    0  &  1
\end{pmatrix}, \ \text{where}\ \theta_1, \theta_2 \in [0, 1).  \label{defreprho}   
\end{equation}
It follows that $\rho(\gamma_0)=-E$ and $\tau(\gamma_0)=-1$. According to Lemma \ref{degnonzero}, $\rho$ defines a rank $2$ bundle $p_{\ast}^{\widetilde{\pi}} (\widetilde{E})$ of degree $-1$. In particular, it is pointed out by Donaldson \cite[p.~579]{Donaldson2021} (see also \cite[Proposition 9.1]{NS1965}) that if $\theta_1=\theta_2=0$, $\rho$ in (\ref{defreprho}) corresponds to the `canonical' indecomposable vector bundle $E(2, -1)$ in \cite[Theorem 10]{Atiyah}. It follows from \cite{Atiyah} that any indecomposable vector bundle $E$ of rank $2$ and degree $-1$ can be written as $E=E(2, -1) \otimes L$ where $L \in \operatorname{Pic}^{0} (\Sigma)$. Together with Lemma \ref{degnonzero}, (\ref{defrepsig}) and (\ref{defreprho}) yield all indecomposable rank $2$ bundles of degree $-1$ on the pointed elliptic curve $(\Sigma, x_0)$.

\begin{proposition}\label{n2d-1noholo}
Let $(\Sigma, x_0)$ be a pointed elliptic curve and $E=p_{\ast}^{\widetilde{\pi}} (\widetilde{E})$ with $\sigma$ and $\rho$ defined in (\ref{defrepsig}) and (\ref{defreprho}). Let $X$ denote the total space of $\pi: E \rightarrow \Sigma$. Then $\mathcal{O}(X)$ can be determined in an explicit way: its lift admits a Taylor expansion \eqref{lift_f_taylor}. Let $\mathcal{A}_k$ denote the homogeneous degree-$k$ part in \eqref{lift_f_taylor}. Then a basis of $\mathcal{A}_k$ can be constructed explicitly if we choose $\theta_1=\theta_2=0$ in (\ref{defreprho}). Let $\{Q_1, Q_2, Q_3\}$ be the basis of $\mathcal{A}_2$ as in \eqref{A2_basis}. Then $Q \coloneqq (Q_1, Q_2, Q_3)$ defines a holomorphic proper map from $X$ to $\mathbb{C}^3$. In particular, $X$ is holomorphically convex.
\end{proposition}

\begin{proof}[Proof of Proposition \ref{n2d-1noholo}]

We consider the branched covering $p: \mathbb{H} \rightarrow \Sigma$ in Lemma \ref{orbifold1}. Let $B \coloneqq p^{-1}(x_0)$ be the branching locus. For any given $f \in \mathcal{O}(X)$, the lift $\widetilde{f}$ is holomorphic on $(\mathbb{H} \setminus B)  \times \mathbb{C}^2$. Let $e_1$ and $e_2$ be the constant frame on the trivial bundle $H \times \mathbb{C}^2$, and let $v_1, v_2$ be the corresponding coordinates. Then
\begin{equation}
    \widetilde{f}(z, v_1, v_2)=\sum_{k=0}^{+\infty} \sum_{t=0}^k f_{k-t,t}(z)v_1^{k-t} v_2^t.  \label{lift_f_taylor}
\end{equation} Here each $f_{k-t, t} \in \mathcal{O}(\mathbb{H} \setminus B)$. 

Near each $y \in B$, we pick $(V, w)$ and $(U, \xi)$ two holomorphic charts centered at $y$ and $x_0$, respectively. Then $p$ is locally expressed as $\xi=w^2$. Since $\pi_{y}$, the stabilizer subgroup at $y$, acts by $w \rightarrow -w$, any local $\widetilde{\pi}$-invariant section of $\mathbb{H} \times \mathbb{C}^{2}$ satisfies $s(-w)=-s(w)$. From the Taylor expansion with respect to $w$, $s(w)=wa(w^2)$ for a holomorphic function $a \in \mathcal{O}(V)$.  Consider the vector bundle $\pi: E \rightarrow \Sigma$. We are led to choose $we_1$ and $we_2$ as local frame on $E$, and the corresponding coordinates $\epsilon_1$ and $\epsilon_2$. Then on $\pi^{-1}(U)$, we write the Taylor expansion of $f \in \mathcal{O}(X)$ in terms of $\xi, \epsilon_1$, and $\epsilon_2$.
\begin{equation}
    f(\xi, \epsilon)=\sum_{k=0}^{+\infty} \sum_{t=0}^k c_{k-t,t}(\xi)\epsilon_1^{k-t} \epsilon_2^t,  \label{f_taylor_E}
\end{equation} where each $c_{k-t,t}$ is holomorphic near $\xi=0$. Note that the bundle coordinates $\epsilon_i$ and the coordinates $v_i$ on the original $\widetilde{\pi}$-bundle are related by $w\epsilon_i=v_i$ for $i=1, 2$. Comparing \eqref{lift_f_taylor} and \eqref{f_taylor_E}, we get that $w^k f_{k-t, t}(w)=c_{k-t, t}(\xi)$ is a holomorphic function of $\xi=w^2$ near $y \in B$.

For any $(z, v_1, v_2) \in (\mathbb{H}\setminus B) \times \mathbb{C}^2$, $f$ satisfies
\begin{align}
    f(z,v_1,v_2)=f(\sigma(\gamma_1)(z), e^{2\pi \theta_1 i }v_2, e^{2\pi \theta_1 i } v_1)=f(\sigma(\gamma_2)(z), -e^{2\pi \theta_2 i }v_1, e^{2\pi \theta_2 i } v_2).     \label{ffinv}
\end{align}
By \eqref{lift_f_taylor}, (\ref{ffinv}), and (\ref{defrepsig}), we get that for $z \in \mathbb{H} \setminus B$
\begin{align}
&f_{k-t,t}(z)=e^{2\pi \theta_1 k i } f_{t,k-t}(\alpha(z))=e^{2\pi \theta_2 k i } (-1)^{k-t} f_{k-t, t}(\beta(z)), \label{1stfinv} \\
&f_{t, k-t}(z)=e^{2\pi \theta_1 k i } f_{k-t, t}(\alpha(z))=e^{2\pi \theta_2 k i } (-1)^t f_{t, k-t}(\beta(z)). \label{2ndfinv}
\end{align}

To sum up, any $f \in \mathcal{O}(X)$ admits a lift $\widetilde{f} \in \mathcal{O} \bigr((\mathbb{H} \setminus B) \times \mathbb{C}^2\bigr)$ such that \eqref{lift_f_taylor} holds with the following properties
\begin{enumerate}
    \item $f_{k-t, t}$ is meromorphic on $\mathbb{H}$, with $\operatorname{ord}_y f_{k-t, t} \geq -k$ for each $y \in B$. Moreover, in terms of coordinates $(V, w)$ at each $y$, we have $w^k f_{k-t, t}(w)$ is a holomorphic function on $w^2$. 
    \item  \eqref{ffinv} holds on $(\mathbb{H} \setminus B) \times \mathbb{C}^2$.
\end{enumerate}

Recall that $\mathcal{A}_k$ is the homogeneous degree-$k$ part $\sum_{t=0}^k f_{k-t,t}(z)v_1^{k-t} v_2^t$ from \eqref{lift_f_taylor}. Our goal is to describe the basis of $\mathcal{A}_k$ in the case \textbf{$\theta_1=\theta_2=0$} explicitly.

Let $H$ be the kernel of the group homomorphism $\psi: \sigma(\widetilde{\pi}) \rightarrow \mathbb{Z}_2$ defined as $\psi(\alpha)=1$ and $\psi(\beta)=0$. We check that $H$, as an index-two subgroup of  $\sigma(\widetilde{\pi})$, is generated by $\alpha^2, \beta$, and $\alpha \beta \alpha^{-1}$. In fact, $H$ is a cocompact Fuchsian group of signature $(1; 2, 2)$. Indeed, if we set
\begin{equation}
  a=\alpha^2,\ \ b=\beta,\ \ \ c=\alpha \beta \alpha^{-1},\ \ \ \delta_1=cb^{-1},\ \ \ \ \delta_2=aba^{-1}c^{-1}.  \label{abc_nota}
\end{equation}
Then $H$ has the presentation as follows.
\begin{align}
    H=\{<a, b, c>\ |\ \ [a, b]\,\delta_1 \delta_2=\delta_1^2=\delta_2^2=1\ \}.  \label{Hwpresent}
\end{align}
By \eqref{defreprho}, $\rho$ induces the following representation on $H$.
\begin{equation}
\rho(a)=\begin{pmatrix}
    1 & 0 \\
    0  &  1
\end{pmatrix}, \ \ \   \rho(b)=-\rho(c)=\begin{pmatrix}
    -1 & 0 \\
    0  &  1
\end{pmatrix},\ \ \  \rho(\delta_1)=\rho(\delta_2)=\begin{pmatrix}
    -1 & 0 \\
    0  &  -1
\end{pmatrix}.
\end{equation} Then $\rho_{H}$ splits as $\chi_1 \oplus \chi_2$ where
\begin{align}
&\chi_1(a)=1, \ \ \chi_1(b)=-1, \ \ \chi_1(c)=1, \ \ \chi_1(\delta_1)=\chi_1(\delta_2)=-1;   \label{chi_1_def} \\
&\chi_2(a)=1, \ \ \chi_2(b)=1, \ \ \chi_2(c)=-1, \ \ \chi_2(\delta_1)=\chi_2(\delta_2)=-1.   \label{chi_2_def}
\end{align}

Note that $\mathcal{A}_0=f_{0, 0}$ is constant. In the following, we determine $\mathcal{A}_k$ by $k=1, 2, 3, \ldots$ in the case $\theta_1=\theta_2=0$.

\textbf{Step 1:}  If $k=1$, $f_{0, 1}(z)=f_{1, 0}(\alpha (z))$ by \eqref{2ndfinv}. It suffices to solve $f_{1, 0}$. We consider the branched covering $p_{H}: \mathbb{H} \rightarrow Y \coloneqq \mathbb{H}/H$. For $i=1, 2$, we consider the direct image sheaf of $\chi_i$-invariant sections of $\mathbb{H} \times \mathbb{C}e_i$ which defines a line bundle $M_i$ on $Y$. Now we consider $L_i$ which is the dual bundle of $M_i$. The pull back of $E^{\ast}$ by the unbranched double covering $q: Y \rightarrow \Sigma$ splits as $L_1 \oplus L_2$ each of which has degree $1$. Note that $q^{\ast}(E^{\ast})$ has degree $2$ and $\operatorname{deg}L_1=\operatorname{deg}L_2$.
Now we observe that $f_{1, 0}$ has a correspondence with $H^0(Y, L_1)$.

To see it, we consider $b_1, b_2 \in q^{-1}(x_0)$. Note that these are two orbifold points on the elliptic curve $Y$. Near $b_i$, we consider $M_1$ has local frame $m=w\,e_1$ where $w$ is the holomorphic coordinate on the upstairs near $p_{H}^{-1}(b_i) \subset B$. Then $L_1$ has local frame $\frac{1}{w}e_1^{\ast}$. Then any holomorphic section of $L_1$ is locally $s=a(u) \frac{1}{w}e_1^{\ast}$ where $a$ is holomorphic with respect to $u=w^2$. After pullback to $\mathbb{H} \times \mathbb{C}$, it becomes $s=\frac{a(w^2)}{w}e_1^{\ast}$. Note that this is exactly the desired asymptotic behavior of $f_{1, 0}$ near $B \subset \mathbb{H}$.

On an elliptic curve $Y$, a degree $1$ line bundle corresponds to $\mathcal{O}_Y(p)$ for a unique point $p \in Y$; see, for example, \cite[Theorem~2.13]{Debarre}. Assume $L_1=\mathcal{O}_Y(p_1)$. As $(f_{1, 0})^2$ becomes $H$-invariant, it descends to a meromorphic function on $Y$ with at most simple poles at $b_1$ and $b_2$. We arrive at $\mathcal{O}_Y(2p_1)=\mathcal{O}_Y(b_1+b_2)$. Using the Jacobi theta function $\vartheta_1$ \cite[Section~2.2]{Debarre}, we may construct a meromorphic function $R$ on $Y$ whose divisor $(R)=2p_1-b_1-b_2$. Note that every zero and pole of $R$ has even order after we pull back on $\mathbb{H}$. We may take a square root of $R$ and write $(A_{1}(z))^2=R(p_H(z))$ for $z \in H \setminus B$. Moreover, $A_1(z)$ satisfies the required invariant condition that $A_1(\gamma z)=\chi_1(\gamma) A_1(z)$ for $\gamma \in H$.

To sum up, we solve the basis of $\mathcal{A}_1$ in an explicit way.
\begin{equation}
\mathcal{A}_1=\operatorname{span}\{A_{1}(z)v_1+A_{1}(\alpha(z))v_2\},\ \ \text{which has dimension } 1.
\end{equation}

\textbf{Step 2:}  If $k=2$, it suffices to solve $f_{2, 0}$ and $f_{1, 1}$.

The analysis is similar to that in Step 1. Now $f_{2, 0}$ is $H$-invariant as
\(
f_{2, 0}(z)=f_{2, 0}(\alpha^2(z))= f_{2, 0}(\beta(z)).
\)
Hence $f_{2, 0}$ descends to a meromorphic function on $Y$, with at most simple poles at both $b_1$ and $b_2$. Therefore $f_{2, 0}$ lies in $H^0(Y, b_1+b_2)$ which has dimension $2$. A convenient choice of its basis can be the constant function $1$ and $A_{2}(z)$. Here $A_2(z)$ is the meromorphic function on $Y$ whose divisor is $2p-b_1-b_2$, where $2p=b_1+b_2$ in the group law on $Y$. Note that the choice of $p$ is not unique, up to $2$-torsion points.

Now $f_{1, 1}$ satisfies
\(
f_{1, 1}(z)=f_{1, 1}(\alpha^2(z))=-f_{1, 1}(\beta(z)).
\)
We introduce a character $\eta: \sigma(\widetilde{\pi}) \rightarrow \{\pm 1\}$ by $\eta(\alpha)=1$ and $\eta(\beta)=-1$, then we may check $\eta(\delta)=1$. As $\eta^2=1$, we conclude that $f_{1, 1}$ descends to a meromorphic section on some $2$-torsion element $N_{\eta} \in \operatorname{Pic}^0(\Sigma)$. As $f_{1, 1}$ has at most a simple pole at $x_0 \in \Sigma$, we conclude that $f_{1, 1}$ corresponds to $H^0(\Sigma, N_{\eta} \otimes \mathcal{O}_{\Sigma}(x_0))$ which has dimension $1$. Once we identify $N_{\eta}=\mathcal{O}_{\Sigma}(x_{\eta}-[0])$, then we have $N_{\eta} \otimes \mathcal{O}_{\Sigma}(x_0)=\mathcal{O}_{\Sigma}([x_{\eta}]-[0]+[x_0])$ which is isomorphic to $\mathcal{O}_{\Sigma}(x_{\eta}+x_0)$ where $x_{\eta}+x_0$ denote the point from the group addition on $\Sigma$. As a meromorphic section of $N_{\eta}$, $f_{1, 1}$ has a simple pole at $x_0$ and a simple zero at $x_0+x_{\eta}$. Then $f_{1, 1}$ is a constant multiple of $B_2(z)$ which can also be expressed by theta functions. 

To sum up, we solve the basis of $\mathcal{A}_2$ as follows.
\begin{equation}
 \mathcal{A}_2=\operatorname{span}\{v_1^2+v_2^2,\ \ \ 
A_{2}(z)v_1^2+A_{2}(\alpha(z))v_2^2,\ \ \ B_2(z)v_1v_2\}.  \label{A2_basis}
\end{equation}

\textbf{Step 3:} We consider the holomorphic map $Q: X \rightarrow \mathbb{C}^3$ defined by
\[
Q(\xi, \epsilon_1, \epsilon_2)=(Q_1, Q_2, Q_3)=(v_1^2+v_2^2, A_{2}(z)v_1^2+A_{2}(\alpha(z))v_2^2, B_2(z)v_1v_2).
\]
As $Q$ is homogeneous with respect to $v_1$ and $v_2$, once we can show that $Q$ is proper, it follows that $X$ is holomorphically convex. Note that each component of $Q$ is indeed holomorphic, in view of \eqref{lift_f_taylor}, \eqref{f_taylor_E}, and $v_i=w \epsilon_i$. We need to show that $Q^{-1}(0)$ is the zero section $\Sigma$, i.e., $v_1=v_2=0$. Assume $Q(z, v_1, v_2)=0$ for some $(z, v_1, v_2) \in (\mathbb{H} \setminus B) \times \mathbb{C}^2$. Our goal is to show $v_1=v_2=0$.

From $Q_1=v_1^2+v_2^2=0$, we see both $v_1$ and $v_2$ vanish as long as either one of them is zero. 
If $v_1v_2 \neq 0$, we have 
\begin{equation}
A_2(z)=A_2(\alpha (z)),\ \ \ \text{and }\ B_2(z)=0.   \label{contra_1}  
\end{equation}
We will show that \eqref{contra_1} leads to a contradiction. 

It is more convenient to work on the elliptic curve $Y$ we introduced in Step 1. Recall $p_{H}: \mathbb{H} \rightarrow Y$, we observe that $\alpha \in \operatorname{Aut}(\mathbb{H})$ descends to an involution $\iota$ on $Y$, i.e., $\operatorname{Id} \neq \iota \in \operatorname{Aut}(Y)$ and $\iota^2=\operatorname{Id}$. As $q: Y \rightarrow \Sigma$ is a double covering, we also have $q \circ \iota=q$. As $b_1, b_2 \in q^{-1}(x_0)$, we may pick $b_1=0$ and $b_2=t$ where $t$ is a $2$-torsion point on $Y$. Then $\iota(y)=y+t$ for any $y \in Y$. From \eqref{contra_1}, we have $A_2(y)=A_2(\iota(y))$ for some $y \in Y$. Consider the meromorphic function $A_2-A_2(y)$ on $Y$, the pole divisor $[0]+[t]$ is linearly equivalent to the zero divisor $[y]+[\iota(y)]$. By the group addition on $Y$, we have $[y]+[y+t]=[0]+[t]$. Hence $2y=0$.

Next we recall the construction of $B_2$ in Step 2. Originally $B_2$ is a meromorphic section on $\Sigma$, with its single zero at $x_0+x_{\eta}$ where $x_{\eta}$ is a $2$-torsion point on $\Sigma$, and a single pole at $x_0$. For simplicity, we choose the defining lattice of $\Sigma$ and $Y$ as $\operatorname{span}_{\mathbb{Z}} \{1, \tau\}$ and $\operatorname{span}_{\mathbb{Z}} \{2, \tau\}$, respectively. Moreover, we may choose $x_0$ as the origin of $\Sigma$, $t=1$, and $q([y]_{Y})=[y]_{\Sigma}$. In the meantime, as $B_2$ corresponds to the character $\eta$ in Step 2, a consideration from the Jacobi theta function shows that $x_{\eta}=\frac{1}{2}$. Now we consider the pull back of $B_2$ on $Y$, with the lifted zeros as $\frac{1}{2}$ and $\frac{1}{2}+1$. By the second equation $B_2(y)=0$ in \eqref{contra_1}, we have $2y=1$ which contradicts $2y=0 \in Y$.

Assume $Q(y, v_1, v_2)=0$ for some $(y, v_1, v_2) \in B \times \mathbb{C}^2$. Recall the holomorphic coordinates near the branching locus are $(\xi, \epsilon_1, \epsilon_2)$ with $\xi=w^2$ and $v_i=w\epsilon_i$ for $i= 1, 2$; see \eqref{f_taylor_E}. We assume that
\begin{align*}
&A_2(z(w))=\frac{a_{-2}}{w^2}+O(1),\ \ \ A_2(\alpha(z(w)))=\frac{\widehat{a}_{-2}}{w^2}+O(1),\ \\   &B_2(z(w))=\frac{b_{-2}}{w^2}+O(1),\ \ \text{where none of}\ a_{-2}, \widehat{a}_{-2}, b_{-2}=0.    
\end{align*}
Then the corresponding $Q$ at $y \in B$ can be solved
\[
Q(y, v_1, v_2)=(0, a_{-2}\epsilon_1^2+\widehat{a}_{-2}\epsilon_2^2, b_{-2}\epsilon_1 \epsilon_2)=(0, 0, 0).
\]
As $a_{-2} \widehat{a}_{-2} b_{-2} \neq 0$, we get $\epsilon_1=\epsilon_2=0$, which implies the corresponding point on $X$ lies on the zero section.

We may carry out the analysis in Steps 1 and 2 to $\mathcal{A}_k$ with $k \geq 3$. However, the corresponding calculations become rather tedious.

\end{proof}

\begin{remark}
As in Remark \ref{Iena_relation}, there is a much simpler method to solve $\operatorname{dim} \mathcal{A}_k$ for any $k \geq 1$. For such a fixed $k$, $\{c_{k-t,t}(\xi)\}_{t=0}^k$ in \eqref{f_taylor_E} defines an element in $H^0\bigl(\Sigma,\operatorname{S}^k (E^{\ast})\bigr)$. Note that $h^0\bigl(\Sigma,\operatorname{S}^k (E^{\ast})\bigr)$ can be computed using Riemann--Roch theorem once we know $h^1\bigl(\Sigma,\operatorname{S}^k (E^{\ast})\bigr)=0$. The motivation of Proposition \ref{n2d-1noholo} is to study an explicit representation of the basis of $\mathcal{A}_k$ from the function theory of elliptic curves. For any indecomposable vector bundle of rank $n \geq 2$ and degree $-n<d <0$ with coprime $n$ and $d$, a representation $\rho$ (as in (\ref{defreprho})) which corresponds to $E(n, d)$ in \cite[Theorem 10]{Atiyah} is given on \cite[p.~579]{Donaldson2021}. It is possible to follow the proof of Proposition \ref{n2d-1noholo} to study the explicit representation of the corresponding $\mathcal{O}(X)$. 
\end{remark}

In general, for any given integers $n$ and $d$, any indecomposable vector bundle of rank $n$ and degree $d$ over an elliptic curve is semistable, and is stable if and only if $n$ and $d$ are coprime, see \cite{Atiyah} and \cite[Appendix A]{Tu}. It seems natural to ask whether there is any connection between the stability of vector bundles and the function theory and K\"ahler geometry on the corresponding total spaces. For example, we propose a special case of Question \ref{queintro}.

\begin{question}
Consider any indecomposable vector bundle $E$ of rank $n$ and degree $d$ over an elliptic curve with $n \geq 2$ and $d \neq 0$. Let $X$ be the total space of $E$. Can we determine $\mathcal{O}(X)$? What is the `best' complete K\"ahler (Hermitian) metric on $X$?
\end{question}

\bibliographystyle{amsplain}

\bibliography{elliptic_bundle}

\end{document}